\documentclass[11pt]{amsart}
\usepackage{indentfirst,latexsym,bm}
\usepackage{amsfonts}
\usepackage{amssymb}
\usepackage{amsmath}
\usepackage{hyperref}
\usepackage{dsfont}
\usepackage{leftidx}
\usepackage{amsthm}
\usepackage{geometry}

\usepackage{color,xcolor}
\usepackage[all]{xy}

\begin{document}
 \thispagestyle{empty}
%\vspace*{2in}
\title[Cocycle deformations and Galois objects]
{Cocycle deformations and Galois objects of semisimple Hopf algebras of dimension $16$}
\author[R. Xiong]{Rongchuan Xiong}
\address{Rongchuan Xiong: School of Mathematical Sciences, East China Normal University, Shanghai 200241, China}
\email{rcxiong@foxmail.com}
\author[Z. Yu]{Zhiqiang Yu$^*$}
\address{Zhiqiang Yu: School of Mathematical Sciences, East China Normal University, Shanghai 200241, China}
\email{zhiqyu-math@hotmail.com}
\thanks{$^*$Corresponding author. The authors were partially supported by the NSFC (Grant No.11771142).}
\subjclass[2010]{16T05; 16S35; 18D10}
%\date{\today}
\maketitle

\newtheorem{thm}{Theorem}[section]
\newtheorem{pro}[thm]{Proposition}
\newtheorem{lem}[thm]{Lemma}
\newtheorem{cor}[thm]{Corollary}
\theoremstyle{definition}
\newtheorem{defi}[thm]{Definition}
\newtheorem{Example}[thm]{Example}
\newtheorem{rmk}[thm]{Remark}

\newcommand{\A}{\mathcal{A}}
\newcommand{\B}{\mathcal{B}}
\newcommand{\C}{\mathcal{C}}
\newcommand{\D}{\mathcal{D}}
\newcommand{\M}{\mathcal{M}}
\newcommand{\K}{\mathds{k}}
\newcommand\Rep{\operatorname{Rep}}
\newcommand\Corep{\operatorname{Corep}}
\newcommand{\Z}{\mathbb{Z}}

\newcommand{\ULa}{\underline{\mathbf{a}}}
\theoremstyle{plain}
\newcounter{maint}
\renewcommand{\themaint}{\Alph{maint}}
\newtheorem{mainthm}[maint]{Theorem}

\theoremstyle{plain}
\newtheorem*{proofthma}{Proof of Theorem A}
\newtheorem*{proofthmb}{Proof of Theorem B}
\begin{abstract}
In this article, we determine   cocycle deformations and Galois objects of non-commutative and non-cocommutative semisimple Hopf algebras of dimension $16$. We show that these Hopf algebras are pairwise twist inequivalent mainly by  calculating their higher Frobenius-Schur indicators, and that except three Hopf algebras which are cocycle deformations of dual group algebras, none of them admit non-trivial cocycle deformations.

\bigskip
\noindent {\bf Keywords:} Semisimple Hopf algebras; Galois objects; Cocycle deformations; Frobenius-Schur indicators.
\end{abstract}
\section{Introduction}
Let $\K$ be an algebraically closed field, $\K^{\times}:=\K\backslash {\{0}\}$, and $H$   a finite-dimensional Hopf algebra over $\K$. Denote by  $\Rep(H)$ and $\Corep(H)$ the category of finite-dimensional $H$-modules and $H$-comudules,  respectively, by  $\mathbb{Z}$   the  set of integer numbers, $\Z_+$   the set of non-negative integers, $\Z_n:=\Z/n\Z$ for any positive integer $n$.

In 1987, Ulbrich \cite{Ul} showed  that isomorphism classes of right Galois objects of $H^*$ are in bijective correspondence with isomorphism classes of fiber functors of the finite tensor category $\Rep(H)$.  In practice, however, it is a challenging question to find all fiber functors. With the development of  tensor category theory, especially fusion category theory \cite{ENO1,ENO2,EGNO,O1}, one can instead study  rank one semisimple module categories over a fusion category $\mathcal{C}$, since semisimple module categories of rank one over  $\mathcal{C}$ correspond to fiber functors on $\mathcal{C}$.

In 2003, Ostrik \cite{O1} gave an abstract and complete description of the exact indecomposable module categories  $\mathcal{M}$ over a fusion category $\mathcal{C}$, that is, $\mathcal{M}$ is equivalent to the category $\mathcal{C}_A$ of right $A$-modules in $\mathcal{C}$ for some exact indecomposable algebra $A$ in $\mathcal{C}$. In particular, Ostrik \cite{O2} classified the exact indecomposable module categories over any group-theoretical fusion category except for an incomplete criterion of equivalence classes. Recently, the complete classification is given by Natale \cite{Na4}. As a byproduct,  fiber functors on a group-theoretical fusion category were determined in \cite[Corolllary\;3.4]{O2} and \cite[Theorem\;1.1]{Na4}.

Natale  showed that semisimple Hopf algebras $H$ obtained by abelian extensions are group-theoretical, and also characterized explicitly the categories $\Rep(H)$ \cite{Na3}. Motivated by the works mentioned above, the authors in \cite{CMNVW} determined   Galois objects and cocycle deformations of non-commutative  non-cocommutative semisimple Hopf algebras of dimension $p^3$ and $p^2q$ with $p,q$ distinct prime numbers, which can be obtained by abelian extensions \cite{Ma1,Na1,Na2,ENO2}. Masuoka has already studied Galois objects and cocycle deformations of four families of semisimple Hopf algebras denoted by $\widehat{\mathcal{D}}_{4n}$, $\widehat{\mathcal{T}}_{4m}$, $\mathcal{A}_{4m}$ and $\mathcal{B}_{4m}$ in a different way, see \cite{Ma2}. The results on Galois objects  can be recovered in principle by using the techniques in \cite{CMNVW}.

In this article, we study Galois objects and cocycle deformations of non-commutative  non-cocomutative semisimple Hopf algebras of dimension $16$, which were classified by Kashina \cite{Ka}. It seems possible to determine  Galois objects of them following the method of \cite{CMNVW}, as they are obtained by abelian extensions and hence are group-theoretical. The results on Galois objects are listed as follows:
\begin{mainthm}[Theorem\;\ref{thmAmain}]
Let $H$ be a semisimple Hopf algebra of dimension $16$ listed in \cite[Table\,1]{Ka}. Then the number of Galois objects of $H^*$ is given as follows:
\begin{enumerate}
\item  $(H_{C:\sigma_1})^*\cong H_{C:\sigma_1} $, $(H_{{c:\sigma_{1}}})^*$, $(H_{B:1})^*\cong H_{C:1}$ and $(H_{B:X})^*\cong H_E$ only have one  trivial Galois object.
\item  $(H_E)^*\cong H_{B:X}$ and $(H_{{c:\sigma_{0}}})^*$ have two Galois objects.
\item  $(H_{C:1})^*\cong H_{B:1}$ and $(H_{a:y})^*$ have  three Galois objects.
\item  $(H_{d:-1,1})^*\cong H_{d:-1,1}$ has four Galois objects.
\item $(H_{d:-1,-1})^*\cong H_{{c:\sigma_{1}}}$ and $(H_{d:1,-1})^*\cong H_{{b:1}}$  have   five Galois objects.
\item  $(H_{{b:x^2y}})^*$ and $(H_{b:y})^*$ have seven Galois objects.
\item $(H_{a:1})^*$ and $(H_{b:1})^*$ have one or two Galois objects, the number of Galois objects of $(H_{d:1,1})^*\cong H_{d:1,1}$ ranges from six to eight.
 \end{enumerate}
\end{mainthm}

If $A$ is a right $H$-Galois object, then up to isomorphism, there exists a unique Hopf algebra $L(A,H):=(A^{op}\otimes A)^{co H}$ such that $A$ becomes an $(L(A,H), H)$-biGalois object \cite{S1}. Moreover, the connection between the  cocycle deformations of Hopf algebras and Galois objects has been established by Schauenburg \cite{S1}. It is of interest to determine all   cocycle deformations of $H$, or equivalently to determine all Drinfeld twists of $H^*$.

In general, it is not easy to determine the Drinfeld twist equivalence classes of semisimple $($or generally quasi-$)$ Hopf algebras. In our cases,  the Frobenius-Schur indicator is a useful tool,  as it is invariant under gauge equivalences \cite{MN,NS1}.  The Frobenius-Schur indicator was first introduced by Linchenko and Montogmery \cite{LM} for Hopf algebras,   then studied extensively  by Kashina, Sommerh\"{a}uer and Zhu \cite{KSZ} for semisimple Hopf algebras, by Mason and Ng \cite{MN}, Ng and Schauenburg \cite{NS1} for semisimple quasi-Hopf algebras, and by Ng and Schauenburg for spherical fusion categories \cite{NS2}. In particular, we are able to obtain the following results mainly by  computing  higher Frobenius-Schur indicators.
\begin{mainthm}[Theorem \ref{thmBmain}]
The Hopf algebras $H_{B:1},~H_{B:X},~H_{d:1,1}$ have only one non-trivial cocycle deformation, and other Hopf algebras listed in \cite[Table\,1]{Ka} admit only trivial cocycle deformations.
\end{mainthm}
It should be pointed out that  the results related to $H_{B:1}$ and $H_{C:\sigma_1}$ in Theorems \ref{thmAmain} and \ref{thmBmain}  were first obtained by Masuoka \cite{Ma2},  as $\mathcal{A}_{16}\cong H_{B:1}$ and $\mathcal{B}_{16}\cong H_{C:\sigma_1}$.

The paper is organized as follows. In section \ref{secPrelinminaries}, we recall some basic notations and properties of group-theoretical fusion categories, Galois objects and cocycle deformations, abelian extensions and Frobenius-Schur indicators. In section \ref{secSemisimple16}, we compute all matched pairs of groups corresponding to the Hopf algebras listed in  \cite[Table 1]{Ka}, and then compute the number of   Galois objects of them. In section \ref{secCocycle}, we mainly use the Frobenius-Schur indicators to find all the  twist equivalence classes, or equivalently to determine all the cocycle deformations of them.
\section{Preliminaries}\label{secPrelinminaries}
We recall some basic knowledge of group-theoretical fusion categories, Galois objects and cocycle deformations, higher Frobenius-Schur indicators and abelian extensions of semisimple Hopf algebras.
\subsection{Group-theoretical fusion categories}
We follow \cite{ENO1,EGNO,O1,O2} (also \cite{CMNVW}) to introduce some necessary concepts and notations.
Let $\C$ be a fusion category over $\K$, $\M$ be an  indecomposable  exact left $\C$-module category, $\text{End}(\M)$ be the tensor category of endofuntors of $\M$ and $\C_{\M}^*$ be the fusion category of $\C$-module  functors of $\M$. The \emph{rank} of $\mathcal{M}$ is the number of isomorphism classes of simple objects of $\mathcal{M}$. Observe that there is a bijective correspondence between structures of a $\C$-module category on $\M$ and monoidal functors $F:\C\rightarrow \text{End}(\M)$. Then exact $\C$-module categories $\M$ of rank one correspond to fiber functors $F:\C\rightarrow  \text{Vec}\cong\text{End}(\M)$.

A fusion category $\C$ is \emph{pointed} if all the simple objects  of $\C$ are invertible, then the isomorphism classes of simple objects form a finite group $G$. Therefore  $\mathcal{C}$ is equivalent to $Vec^{\omega}_{G}$, the category of finite-dimensional $G$-graded vector spaces with the associativity isomorphism given by a normalized 3-cocycle $\omega\in Z^3(G,\K^{\times})$.

Let $\C$ and $\mathcal{D}$ be two fusion categories. $\C$ and $\mathcal{D}$ are \emph{categorically Morita equivalent} if there exists an exact indecomposable $\C$-module category $\M$ such that $\mathcal{D}^{\text{op}}\cong \C_{\M}^*$ as fusion categories.
\begin{defi} A fusion category is said to \emph{group-theoretical} if it is categorically Morita equivalent to a pointed fusion category $Vec^{\omega}_{G}$.
\end{defi}
The exact indecomposable module categories over a group-theoretical fusion category were classified by Ostrik \cite{O2}: Every exact indecomposable module category over the category $Vec^{\omega}_{G}$ is determined by a pair $(F, \alpha)$, where $F$ is a subgroup of $G$ such that the class of $\omega|_{F\times F\times F}$ is trivial in $H^3(F, \K^{\times})$ and $\alpha: F\times F\rightarrow \K^{\times} $ is a $2$-cochain on $F$ satisfying $d^2(\alpha)=
\omega|_{F\times F\times F}$. Denote by $\mathcal{C}(G, \omega, F,\alpha)$ the corresponding group-theoretical fusion category. Every exact indecomposable module category over $\mathcal{C}(G, \omega, F,\alpha)$ is determined by a pair $(L,\beta)$, where $L$ is a subgroup of $G$ and $\beta$ is a 2-cochain on $L$ such that $d^2(\beta)=\omega|_{L\times L\times L}$.

A $2$-cocycle $\alpha\in Z^2(G,\K^{\times})$ is said to be \emph{non-degenerate} if the twisted group algebra $\K_{\alpha}[G]$ is a simple  algebra. The following result plays a key role in \cite{CMNVW} to determine the Galois objects of semisimple Hopf algebras of dimension $p^3$ and $p^2q$.
\begin{thm}$($\cite[Corollary 3.4]{O2}, \cite[Theorem 1.1]{Na4}$)$ \label{fiberfuncondition} Fiber functors on $\mathcal{C}(G, \omega, F,\alpha)$
correspond to pairs $(L,\beta)$, where $L$ is a subgroup of $G$ and $\beta$ is a $2$-cocycle on $L$,
such that the following conditions are satisfied:
\begin{enumerate}
  \item The class of $\omega|_{ L\times L\times L}$ is trivial;
  \item $G = LF$;
  \item The class of the $2$-cocycle $\alpha|_{F\cap L}\beta^{-1}|_{L\cap F}$ is non-degenerate.
\end{enumerate}
Two fiber functors $(L,\beta)$, $(L{'},\beta{'})$ are isomorphic if and only if
there exists an element $g\in G$ such that $L{'}=gLg^{-1}$, and the cohomology class of
the $2$-cocycle $\beta^{-1}\beta'^g\Omega_g$ is trivial in $H^2(L, \K^{\times})$, where $\beta'^g(h,l):=\beta'(ghg^{-1},glg^{-1})$ for $h,l\in L$ and $\Omega_g(a,b):=\frac{\omega(gag^{-1},gbg^{-1},g)\omega(g,a,b)}{\omega(gag^{-1},g,b)}$ for $a,b\in G$.
\end{thm}

\begin{rmk}\label{equivalencefiber}From \cite[section 3.1]{Na4},  for any $g\in G$, we have $d^2(\Omega_g)=\frac{\omega}{\omega^g}$, where $\omega^g(a,b,c):=\omega(gag^{-1},gbg^{-1},gcg^{-1})$, $\forall a, b, c\in G$. If  the class $\omega|_{ L\times L\times L}$ is trivial and $L$ is normal, then $[\Omega_g|_{L\times L}]\in H^2(L,\K^{\times})$.
\end{rmk}
\subsection{Cocycle deformations and twist deformations}
\begin{defi} [\cite{EV}] Let $H$ be a Hopf algebra over $\K$. A twisting $H_{\Omega}$ of  $H$ is a Hopf algebra with the same algebra structure and counit $\varepsilon$,  with the coproduct and antipode given by
\begin{gather*}
\Delta_{\Omega}(h)=\Omega \Delta(h)\Omega^{-1},\quad S_{\Omega}(h)=uhu^{-1},\quad\forall h\in H,
\end{gather*}
 where $u\in H$ and $\Omega\in H\otimes H$ are invertible elements.
\end{defi}

$H_{\Omega}$ is well-defined if and only if $\Omega\in H\otimes H$ is a $2$-pseudo-cocycle, that is, $\partial_2(\Omega)$ lies in the centralizer of $(\Delta\otimes \text{id})\Delta(H)$ in $H\otimes H\otimes H$, where
\begin{gather*}
\partial_2(\Omega):=(id\otimes \Delta)(\Omega^{-1})(1\otimes \Omega^{-1}) (\Omega\otimes1)( \Delta\otimes id)(\Omega).
\end{gather*}
$\Omega$ is a $2$-cocycle if and only if  $\partial_2(\Omega)=1\otimes1\otimes1,(\varepsilon\otimes id)\Omega=(id\otimes\varepsilon)\Omega=1$. In this case, we call $\Omega$ a \emph{Drinfeld twist} of $H$ denoted by $J$ and $H^{J}:=H_{\Omega}$ is a \emph{(Drinfeld) twist  deformation} of $H$.

Let $H$, $H^{\prime}$ be two Hopf algebras. Then $H$ and $H^{\prime}$  are said to be \emph{twist equivalent}, if $H^{\prime}\cong H^J$ for some Drinfeld twist $J$ of $H$. Furthermore, if  $H$ and $H^{'}$ are finite-dimensional, then they are twist equivalent if and only if $\Rep(H)\cong \Rep(H^{\prime})$ as tensor categories \cite{S1}.

The properties like the semisimplicity or   Grothendieck rings of Hopf algebras are preserved under twisting deformations. Moreover,
\begin{thm}$($\cite[Theorem 4.1]{Ni}$)$ \label{twistings}Let $H$ and $L$ be Hopf algebras over $\K$. Denote by $Gr(H)$ and $Gr(L)$ the Grothendieck rings of $H$ and $L$, respectively. Then $Gr(H)\cong Gr(L)$ as fusion rings if and only if $H$ is a twisting of $L$.
\end{thm}

Now we introduce the dual concept of the Drinfeld twist. Let $H$ be a Hopf algebra and $\sigma:H\otimes H\rightarrow \K$ be a map. Then $\sigma$ is said to be a \emph{Hopf $2$-cocycle} of $H$, if it is a convolution invertible morphism satisfying
\begin{gather*}
\sigma(x,1)=\varepsilon(x)=\sigma(1,x),\quad
\sigma(x_{(1)},y_{(1)})\sigma(x_{(2)}y_{(2)},z)=\sigma(y_{(1)},z_{(1)})\sigma(x,y_{(2)}z_{(2)}),\quad\forall x,y,z\in H.
\end{gather*}

Given a $2$-cocycle $\sigma : H\otimes H\rightarrow \K$. Denote by $\sigma^{-1}$ the convolution inverse of $\sigma$. We can construct a new Hopf algebra $(H^{\sigma},m^{\sigma},1,\Delta,\epsilon, S^{\sigma})$, where $H^{\sigma}=H$ as coalgebras, and
\begin{gather*}
m^{\sigma}(x\otimes y)=\sigma(x_{(1)},y_{(1)})x_{(2)}y_{(2)}\sigma^{-1}(x_{(3)},y_{(3)}), \\
S^{\sigma}(x)=\sum\sigma(x_{(1)},S(x_{(2)}))S(x_{(3)})\sigma^{-1}(S(x_{(4)}),x_{(5)}),\quad\forall x,y\in H.
\end{gather*}
\begin{defi}
We call $H^{\sigma}$ the  cocycle deformation of $H$ by $\sigma$.
\end{defi}

%The Hopf $2$-cocycle deformation $H^{\sigma}=H$ as $\K$-coalgebras, with the multiplication and the antipode given by
%\begin{gather*}
%x\cdot^{\sigma}y=\sigma(x_{(1)},y_{(1)})x_{(2)}y_{(2)}\sigma^{-1}(x_{(3)},y_{(3)}),\\
%S_{\sigma}(x)=\sum\sigma(x_{(1)},S(x_{(2)}))S(x_{(3)})\sigma^{-1}(S(x_{(4)}),x_{(5)}),\quad\forall x,y\in H.
%\end{gather*}

A $2$-cocycle deformation $H^{\sigma}$ of $H$ by $\sigma$ is said to be  \emph{trivial} if $H^{\sigma}\cong H$ as Hopf algebras.
\begin{rmk}[\cite{S1}]
Given a Hopf $2$-cocycle $\sigma:H\otimes H\rightarrow \K$, there exists an algebra $H_{\sigma}$ with the multiplication given by:
\begin{gather*}
x\cdot_{\sigma}y=\sigma(x_{(1)},y_{(1)})x_{(2)}y_{(2)},\quad x,y\in H.
\end{gather*}
Moreover, $\Delta:H_{\sigma}\rightarrow H_{\sigma}\otimes H$ is an algebra map and then $H_{\sigma}$ becomes a right $H$-comodule algebra and $(H_{\sigma})^{\text{co}\;H}\cong\K$.
\end{rmk}
\subsection{Hopf Galois objects}
\begin{defi}
Let $H$, $L$ be Hopf algebras over $\K$. A right $H$-Hopf comodule algebra $A$ is a right $H$-Galois object if $A^{co H}=\K$ and the Hopf-Galois map $\beta:A\otimes A\rightarrow A\otimes H$, $\beta(a\otimes b)=ab_{(0)}\otimes b_{(1)}$, is bijective. A left $L$-Galois object is defined analogously. An $(L,H)$-biGalois object $A$ is a  right $H$-Galois object and a left $L$-Galois object, and the left $L$-comodule and right $H$-comodule morphisms are compatible.
\end{defi}
Given a right $H$-Galois object $A$, there is a Hopf algebra $L:=L(A,H)$ attached to the pair $(A,H)$ such that $A$ becomes an $(L,H)$-biGalois object \cite{S1}. The Hopf algebra $L(A,H)$ is unique up to isomorphism. A right $H$-Galois object $A$ is \emph{trivial} if $A\cong H$ as right $H$-comodule algebras. In this special case, $L\cong H$ as Hopf algebras.

\begin{rmk}
If $\sigma:H\otimes H\rightarrow \K$ is a Hopf $2$-cocycle, then the algebra $H_{\sigma}$ is an $(H^{\sigma},H)$-biGalois object.
\end{rmk}

Let $H$ and  $L$ be finite-dimensional Hopf algebras over $\K$. Then $H$ and $L$ are   \emph{Morita-Takeuchi equivalent}, if  $\Corep(H)\cong\Corep(L)$ as finite tensor categories.

\begin{thm}$($\cite[Corollary 5.7]{S1}$)$\label{thmS1}
Let $H,L$ be finite-dimensional Hopf algebras. Then $H$ and  $L$ are Morita-Takeuchi equivalent, if and only if  $H$ is a cocycle deformation of $L$, if and only if there exists a $(H,L)$-biGalois object $A$ such that $L\cong L(A,H)$.
\end{thm}
\begin{pro}$($\cite[Theorem 1.7]{Ul}$)$\label{Galoisfiber}
Let $H$ be a finite-dimensional Hopf algebra over $\K$. Then there is a bijective correspondence between the following sets:
\begin{enumerate}
  \item The set of right Galois objects of $H^*;$
  \item The set of fiber functors on $\Rep(H).$
\end{enumerate}
\end{pro}

\subsection{Higher Frobenius-Schur indicators}
\begin{defi}[\cite{LM,KSZ}]
Let $H$ be a semisimple Hopf algebra, $\Lambda\in H$ be an integral such that $\varepsilon(\Lambda)=1$, $V$ be a finite-dimensional  representation of $H$, and $\chi$ be its character. For any non-negative integer $n$, the $n$-th Frobenius-Schur indicator $\nu_{n}(\chi)$ of $V$ is given by $$\nu_{n}(\chi):=\chi(\Lambda^{[n]}),$$ where $\Lambda^{[n]}:=\Lambda_{(1)}\cdots\Lambda_{(n)}$  (here $\Lambda^{[0]}:=1$) is the $n$-th Sweedler power of $\Lambda$.
\end{defi}
\begin{rmk}The Frobenius-Schur indicator is also well-defined for semisimple quasi-Hopf algebras and spherical fusion categories, see \cite{NS1,NS2} for details.
\end{rmk}

\begin{rmk}
 Let $H$ be a Hopf algebra fitting into an abelian extension : $\K\rightarrow \K^{\Gamma}\rightarrow H\rightarrow \K[F]\rightarrow \K$. Then the integral element $\Lambda$ of $H $ satisfying $\varepsilon(\Lambda)=1$ is given by
\begin{align*}
\Lambda=\frac{1}{\mid F\mid}\Sigma_{g\in F}\delta_1\#g,
\end{align*}
where $\delta_1$ is the orthogonal primitive idempotent satisfying $\langle \delta_1,g\rangle=\delta_{1,g}$, for any $g\in \Gamma$.
\end{rmk}

\begin{thm}$($\cite[Theorem 4.1]{MN}, \cite[Theorem 4.1]{NS1}$)$\label{gaugeinvariant}
The  Frobenius-Schur indicators are invariant under gauge transformations of semisimple (quasi)-Hopf algebras.
\end{thm}

The Frobenius-Schur indicator is a gauge invariant, so it is useful to distinguish whether two semisimple Hopf algebras are twist inequivalent.

\subsection{Abelian extensions of Hopf algebras}
We follow \cite{H,K,Ma3,Na3} to introduce some necessary notations and concepts on abelian extensions $($see also \cite[section 3]{CMNVW}$)$.
Let $F$, $\Gamma$ be finite groups. A \emph{matched pair} $(F,\Gamma,\lhd,\rhd)$ consists of two actions satisfying the following compatible conditions: $\forall s, t\in \Gamma$, $\forall x, y\in F$,
\begin{gather*}
s\rhd (xy) = (s \rhd x)((s \lhd x) \rhd y), \quad
(st)\lhd x = (s \lhd (t \rhd x))(t \lhd x).
\end{gather*}

Let $(F,\Gamma,\lhd,\rhd)$ be a matched pair of groups. The \emph{bicrossed product} $F\bowtie\Gamma$  of $F$ and $\Gamma$  is a group with unit $1 \bowtie 1$ and multiplication given by
\begin{gather*}
(x\bowtie s)(y\bowtie t) = x(s \rhd y)\bowtie(s \lhd y)t, \quad\forall x, y\in F, s, t\in\Gamma.
\end{gather*}

 Obviously, $F\times 1$ and $1\times\Gamma$ are subgroups of $F\bowtie\Gamma$ and isomorphic to $F$ and $\Gamma$, respectively.
If we do not distinguish the subgroups $F\times 1$, $1\times\Gamma$ from $F$, $\Gamma$,
then  $F\bowtie\Gamma$ admits an \emph{exact factorization} $F\bowtie\Gamma=F\Gamma$.

Conversely,
if a finite group $G$ admits an exact factorization $G=F{'}\Gamma{'}$, where $F{'}\cong F$ and $\Gamma{'}\cong\Gamma$ as groups, then any $g\in G$ can be written uniquely as $xs$, where $x\in F,~s\in \Gamma$.
Henceforth, the group multiplication implies $(xs)(yt)=(xy{'})(s{'}t)$, for unique $y{'}\in F,~s{'}\in \Gamma$,  which  gives two actions $\lhd:\Gamma\times F\rightarrow\Gamma$, $(s,y)\mapsto s{'}$, and $\rhd:\Gamma\times F\rightarrow F$ $(s,y)\mapsto y{'}$, such that $(F,\Gamma,\lhd,\rhd)$ is a matched pair. It is easy to see that these actions are determined by the relations
\begin{gather*}
sy =y{'}s{'}= (s \rhd y)(s \lhd y), \quad y \in F, s \in \Gamma.
\end{gather*}

Let $(F,\Gamma,\lhd,\rhd)$ be a matched pair of groups, ${\{e_s}\}_{s\in \Gamma}$ be the dual primitive orthogonal basis in $\K^{\Gamma}$.
The action $\lhd:\Gamma\times F\rightarrow\Gamma$ corresponds to an action  $\rightharpoonup:F\times \K^{\Gamma}\rightarrow \K^{\Gamma}$ so that
\begin{align*}
 a\rightharpoonup e_x=e_{x\lhd a^{-1}},\quad a\in F, x\in\Gamma.
\end{align*}
Let $\sigma:F\times F \rightarrow (\K^{\Gamma})^\times$ be a normalized $2$-cocycle of the group $F$ with coeffients in $(\K^{\Gamma})^\times$. We write $\sigma:=\sum_{g\in \Gamma}\sigma_{g}e_g$. Then
 \begin{gather*}
 \sigma_{s\lhd x}(y, z) \sigma_s(x, yz) = \sigma_s(xy, z) \sigma_s(x, y),\quad
 \sigma_1(x, y) = \sigma_s(x, 1) = \sigma_s(1, y) = 1,\quad \forall x, y, z\in F.
 \end{gather*}

 Dually, let ${\{\delta_x}\}_{x\in F}$ be the dual primitive orthogonal basis in $\K^{F}$; the action $\rhd:\Gamma\times F\rightarrow F$ corresponds to an action $\leftharpoonup:\K^F\times \Gamma\rightarrow \K^F$ so that
 \begin{align*}
  \delta_x\leftharpoonup a=\delta_{a^{-1}\rhd x},\quad a\in\Gamma,x\in F.
  \end{align*}
 Let $\tau=\sum_{x\in F}\tau_x\delta_x:\Gamma\times \Gamma\rightarrow (\K^{F})^{\times}$ be a normalized $2$-cocycle of $\Gamma$ with coefficients in $(\K^{F})^{\times}$. We write $\tau:=\sum_{g\in F}\tau_{g}e_g$. Then
 \begin{gather*}
 \tau_x(st, u) \tau_{u\rhd x}(s, t) = \tau_x(s, tu) \tau_x(t, u),\quad
 \tau_1(s, t) = \tau_x(s, 1) = \tau_x(1, t) = 1, \quad \forall s, t, u\in \Gamma.
 \end{gather*}

Let $\K^{\Gamma}{}^\tau\#_\sigma\K[F]$ denote the vector space $\K^{\Gamma}\otimes \K[F]$ with multiplication and comultiplication given by
\begin{gather*}
(e_g\#x)(e_h\#y) =\delta_{g\lhd x,h} \sigma_g(x, y)e_g\#xy,\\
\Delta(e_g\#x) =\sum_{st=g}\tau_x(s,t)e_s\#(t \rhd x)\otimes e_t\#x, \quad g, h \in \Gamma, x, y \in F.
\end{gather*}
  If $\sigma$ and $\tau$ obey some certain compatible conditions,
then $H:=\K^{\Gamma}{}^{\tau}\#{}_{\sigma}\K[F]$ is a semisimple Hopf algebra fitting into an exact sequence of Hopf algebras $\K\rightarrow \K^{\Gamma}\rightarrow H\rightarrow \K[F]\rightarrow \K$. We say $H$ is an $\textit{abelian
extension}$ associated to the matched pair $(F,\Gamma,\lhd,\rhd)$. Moreover, every Hopf algebra fitting into such an exact sequence can be described in this way. See e.g. \cite[Part 1]{Ma3} for details.

Given a matched pair $(F,\Gamma,\lhd,\rhd)$, denote by $\text{Opext}(\K^\Gamma,\K[F])$ the set of equivalence classes of abelian extensions $\K\rightarrow \K^{\Gamma}\rightarrow \K^{\Gamma}\leftidx{^\tau}{\#_{\sigma}\K[F]}\rightarrow \K[F]\rightarrow \K$, which is a finite group under the Baer product of extensions.
By  \cite[(3.14)]{K}, there is an exact sequence
\begin{align*}
& 0\rightarrow H^1(F\bowtie\Gamma,\K^{\times})\xrightarrow{res}H^1(F,\K^{\times})\oplus H^1(\Gamma,\K^{\times})
 \rightarrow Aut(\K^{\Gamma}\leftidx{^\tau}{\#_{\sigma}\K[F]})\\&\rightarrow H^2(F\bowtie\Gamma,\K^{\times})\xrightarrow{res}H^2(F,\K^{\times})\oplus H^2(\Gamma,\K^{\times})\rightarrow Opext(\K^\Gamma,\K[F])\\&
 \xrightarrow{\overline{\omega}}H^3(F\bowtie\Gamma,\K^{\times})\xrightarrow{res}
H^3(F,\K^{\times})\oplus H^3(\Gamma,\K^{\times})\rightarrow\cdots.
\end{align*}
In particular, the element $[\tau,\sigma]\in Opext(\K^\Gamma,\K[F])$ is mapped under $\overline{\omega}$ onto a 3-cocycle $\omega(\tau,\sigma)\in Z^3(F\bowtie\Gamma,\K^{\times})$, which is defined by:
\begin{align}
\omega(\tau,\sigma)(a,b,c)=\tau_{\pi(c)}(p(a)\lhd \pi(b),p(b))\sigma_{p(a)}(\pi(b),p(b)\rhd\pi(c)),\quad \forall a,b,c \in F\bowtie\Gamma,\label{eq3cocycle}
\end{align}
where $\pi: F\bowtie\Gamma\rightarrow F, g\bowtie h\mapsto g$ and $p: F\bowtie\Gamma\rightarrow\Gamma, g\bowtie h\mapsto h$ are projections.
\begin{thm}$($\cite[Theorem 1.3]{Na3}$)$
Let $(F,\Gamma,\lhd,\rhd)$ be a matched pair of groups and $H$ be a Hopf algebra fitting into an abelian extension $\K\rightarrow \K^{\Gamma}\rightarrow H\rightarrow \K[F]\rightarrow \K$
 associated to $(F,\Gamma,\lhd,\rhd)$. Then $H$ is group-theoretical and  $\Rep(H)\cong \C(F\bowtie\Gamma,\omega(\tau,\sigma),F,1)$ as fusion categories.
\end{thm}
%\subsection{The subgroups of groups of order $16$}
%Let $G$ be a non-abelian group of order $16$. Then $G$ is isomorphic to one of the following groups:
%\begin{description}
%  \item[(1)] $Q_{16}:=\langle a,b\mid a^8=1,b^2=a^4,b^{-1}ab=a^{-1}\rangle$;
%  \item[(2)] $D_{16}:=\langle a,b\mid a^8=1=b^2 ,b^{-1}ab=a^{-1}\rangle$;
%  \item[(3)] $SD_{16}:=\langle a,b\mid a^8=1=b^2 ,b^{-1}ab=a^{3}\rangle$;
%  \item[(4)] $G:=\langle a,b\mid a^8=1=b^2 ,b^{-1}ab=a^{5}\rangle$;
%  \item[(5)] $G:=\langle a,b\mid a^4=1=b^4 ,b^{-1}ab=a^{-1}\rangle$;
%  \item[(6)] $G:=\langle a,b,c\mid a^4=b^2=c^2=1,[a,b]=c,[c,a]=[c,b]=1\rangle$;
%  \item[(7)] $D_8\otimes \Z_2$;
%  \item[(8)] $Q_8\otimes \Z_2$;
%  \item[(9)] $G:=\langle a,b,c\mid a^4=b^2=c^2=1,[b,c]=a^2,[a,b]=[a,c]=1\rangle\cong D_8*\Z_4$.
%\end{description}
%\begin{thm}
%There are seven subgroups of order $8$ of $D_8*\Z_4$: $\langle a\rangle\times \langle c\rangle$, $\langle b,c\rangle$, $\langle ba,c\rangle$, $\langle a\rangle\times\langle b\rangle$, $
%\end{thm}
\section{Galois objects of semisimple Hopf algebras of dimension $16$}\label{secSemisimple16}
We study Galois objects of non-commutative and non-cocommutative  semisimple Hopf algebras of dimension $16$ listed in \cite[Table\,1]{Ka}.

\subsection{Singer pairs and matched pairs}\label{subsectionMatchedpair}
Let $H$ be a semisimple Hopf algebra of dimension $16$ in \cite[Table\,1]{Ka}. Kashina \cite{Ka} showed that $H$ fits into a cocentral abelian exact sequence of Hopf algebras
\begin{gather*}
\K^{\Gamma}\hookrightarrow H\twoheadrightarrow \K[F],
\end{gather*}
where $F:= \langle t\rangle\cong \Z_2$ and $\Gamma$ is a group of order $8$.

Moreover, $H$ is a bicrossed product $\K^{\Gamma}{}^{\rho,\theta}\#_{\rightharpoonup,\sigma}\K[F]$ with a non-trivial action $\rightharpoonup:\K[F]\otimes \K^{\Gamma}\rightarrow \K^{\Gamma}$ and a trivial coaction $\rho:\K[F]\rightarrow \K[F]\otimes \K^{\Gamma}$,  a Singer $2$-cocycle $\sigma:\K[F]\otimes \K[F]\rightarrow \K^{\Gamma}$ and a dual Singer $2$-cocycle $\theta:\K[F]\rightarrow \K^{\Gamma}\otimes \K^{\Gamma}$. The pair $(\K[F], \K^{\Gamma},\rightharpoonup,\rho)$ is called a Singer pair.

There is a bijective correspondence between the Singer pair $(\K[F], \K^{\Gamma},\rightharpoonup,\rho)$ and the matched pair $(F,\Gamma, \lhd,\rhd)$.
%It is given by two correspondences between action $\rhd:\Gamma\otimes F\rightarrow \Gamma$ and $\rightharpoonup:F\otimes k^{\Gamma}\rightarrow k^{\Gamma}$, and between action $\lhd:\Gamma\otimes F\rightarrow F$ and coaction $\rho:F\rightarrow F\otimes k^{\Gamma}$.
In  our cases, the trivial coaction $\rho:\K[F]\rightarrow \K[F]\otimes \K^{\Gamma}$ implies that $\rhd:\Gamma\otimes F\rightarrow \Gamma$ is trivial and the module action $\rightharpoonup: F\otimes \K^{\Gamma}\rightarrow \K^{\Gamma}$ corresponds to the action $\lhd:\Gamma\otimes F\rightarrow F$ so that
\begin{align}
x\rightharpoonup\delta_{s}=\delta_{{s\lhd x^{-1}}},\quad\forall x\in F,~s\in \Gamma.
\end{align}

Given a Singer $2$-cocycle $\sigma:\K[F]\otimes \K[F]\rightarrow \K^{\Gamma}$ and a dual Singer $2$-cocycle $\theta:\K[F]\rightarrow \K^{\Gamma}\otimes \K^{\Gamma}$, the corresponding   element $[\tau,\sigma]\in \text{Opext}(\K^\Gamma,\K[F])$ is given by
\begin{align}
\tau_f(g,h)=\theta(f)(g,h),\quad \forall g, h\in \Gamma, ~f\in F,\\
\sigma_s(x,y)=\sigma(x,y)(s),\quad \forall s\in\Gamma,~x,y\in F.
\end{align}

Let $H$ be a Hopf algebra listed in \cite[Table 1]{Ka}. Then $H\cong \K^{\Gamma}{}^{\rho,\theta}\#_{\rightharpoonup,\sigma}\K[F]$ for some Singer pair $(\K[F], \K^{\Gamma},\rightharpoonup,\rho)$, Singer $2$-cocycle $\sigma$ and dual Singer $2$-cocycle $\theta$. We shall describe the bicrossed product $F\bowtie\Gamma$ and  the $3$-cocycle $\omega(\sigma,\tau)\in Z^3(F\bowtie\Gamma,\K^{\times})$. We proceed this  case by case  following  the classification of Kashina.

\subsection*{Case \uppercase\expandafter{\romannumeral1}: $\Gamma=\mathbb{Z}_{4}\times \mathbb{Z}_{2}=\langle x\rangle\times \langle y\rangle$.}

Let $\{e_{p,q}\}_{\substack{0\leq p\leq3,0\leq q\leq1}}$ be the dual basis in $\K^{\Gamma}$, that is, $\langle e_{p,q},x^{i}y^{j} \rangle=\delta_{p,i}\delta_{q, j}$ for $0\leq i< 4,0\leq j< 2$. In this case, $\theta(t)=\sum_{ijpq}(-1)^{jp}e_{ij}\otimes e_{pq}$. Then $$\tau_{{t^s}}(x^{i}y^{j},x^{k}y^{l})=(-1)^{jks}.$$ Indeed,
$\tau(x^{r}y^{s},x^{k}y^{l})(t)=\theta(t)(x^{r}y^{s},x^{k}y^{l})=\sum_{ijpq}(-1)^{jp}\langle e_{i,j}\otimes e_{p,q},x^{r}y^{s}\otimes x^{k}y^{l}\rangle=(-1)^{sk}$.

\textbf{Case (a)}: The action is given by $t\rightharpoonup e_{i,j}=e_{i+2j,j}$ and $\sigma(t,t)=\sum_{p,q}(-1)^{kp+lq}e_{p,q}$ for $k,l=0,1$. In this case, the Singer pair forms the Hopf algebra $H_{k,l}$.

We claim that $F\bowtie\Gamma$ is isomorphic to
      \begin{align*}
      D_8*\mathbb{Z}_4=G_7:=\langle x,y,t\mid x^4=y^2=t^2=1,[x,y]=[x,t]=1, [y,t]=x^2\rangle.
      \end{align*}
Indeed,   $x\rhd t=y\rhd t=t$, and $x\lhd t=x,~y\lhd t=x^{2}y$ since the action $\rhd$ is trivial and
\begin{align*}
\langle x\lhd t, e_{i,j}\rangle=\langle x, t\rightharpoonup e_{i,j}\rangle=\langle x, e_{i+2j,j}\rangle=\delta_{1,i+2j}\delta_{0,j},\\
\langle y\lhd t, e_{ij}\rangle=\langle y, t\rightharpoonup e_{ij}\rangle=\langle y, e_{i+2j,j}\rangle=\delta_{0,i+2j}\delta_{1,j}.
\end{align*}
Then the claim follows by the following relations
\begin{align*}
(1\bowtie x) (t\bowtie1)= t\bowtie x,\quad (1\bowtie y)(t\bowtie1)=(y\rhd t)\bowtie(y\lhd t)=t\bowtie x^{2}y.
\end{align*}

Now we calculate $\omega(\sigma,\tau)\in Z^3(F\bowtie\Gamma,\K^{\times})$.
\begin{enumerate}
\item $H_{a:1}:=H_{0,0}$. Then  $\sigma$ is trivial, that is, $\sigma_{{x^iy^j}}(t^r,t^s)=1$. Let
 $a=t^qx^iy^j$,  $b=t^rx^ky^l$ and $c=t^sx^my^n$ for some $q,r,s,j,l,n\in\Z_2,i,k,m\in\Z_4$. By the formula \eqref{eq3cocycle}, we have
\begin{align*}
&\omega(\tau,\sigma)(a,b,c)=\tau_{\pi(c)}(p(a)\lhd \pi(b),p(b))\sigma_{p(a)}(\pi(b),p(b)\rhd\pi(c))\\
&=\tau_{t^s}((x^i\lhd t^r)(y^j\lhd t^r),x^ky^l)
=\tau_{t^s}(x^{{i+2rj}} y^j,x^ky^l)=(-1)^{jks}.
\end{align*}

\item $H_{a:y}:=H_{0,1}$. Since $\sigma(t,t)=\Sigma_{p,q}(-1)^{q}e_{p,q}$, it follows that $\sigma_{{x^iy^j}}(t^r,t^s)=(-1)^{jrs}$. Then
  \begin{align*}
&\omega(\tau,\sigma)(a,b,c)=\tau_{\pi(c)}(p(a)\lhd \pi(b),p(b))\sigma_{p(a)}(\pi(b),p(b)\rhd\pi(c))\\
&=\tau_{t^s}( (x^i\lhd t^r)(y^j\lhd t^r),x^ky^l)\sigma_{{x^iy^j}}(t^r,t^s)
=\tau_{t^s}(x^{{i+2rj}} y^j,x^ky^l)(-1)^{jrs}\\
&=(-1)^{jks}(-1)^{jrs}=(-1)^{j(k+r)s}.
\end{align*}
\end{enumerate}
\textbf{Case (b):} The action is given by $t\rightharpoonup e_{ij}=e_{-i,j}$ and $\sigma(t,t)=\sum_{p,q}(-1)^{kp+lq}e_{p,q}$ for $0\leq k,l\leq1$. In this case, the Singer pair forms the Hopf algebra $H_{k,l}$.

We claim that $F\bowtie\Gamma$ is isomorphic to
\begin{align*}
      D_8\times \mathbb{Z}_2= G_8:=\langle x,t,y\mid x^4=t^2=y^2=1,tx=x^{-1}t,[x,y]=[t,y]=1 \rangle.
      \end{align*}
Indeed, $x\rhd t=y\rhd t=t$, $x\lhd t=x^{-1}$ and $y\lhd t=y$ since the action $\rhd$ is trivial and
\begin{align*}
\langle x\lhd t, e_{ij}\rangle=\langle x, t\rightharpoonup e_{ij}\rangle=\langle x, e_{-i,j}\rangle=\delta_{1,-i}\delta_{0,j},\\
\langle y\lhd t, e_{ij}\rangle=\langle y, t\rightharpoonup e_{ij}\rangle=\langle y, e_{-i,j}\rangle=\delta_{0,-i}\delta_{1,j}.
\end{align*}
Then the claim follows by the following relations
\begin{align*}
(1\bowtie x) (t\bowtie1)=(x\rhd t)\bowtie(x\lhd t)=t\bowtie x^{-1},\quad(1\bowtie y)(t\bowtie1)= t\bowtie y.
\end{align*}

Now we calculate $\omega(\sigma,\tau)\in Z^3(F\bowtie\Gamma,\K^{\times})$.
\begin{enumerate}
\item $H_{b:1}:=H_{0,0}\cong H_{1,0}$. Then   $\sigma$ is trivial. Let $a=t^qx^iy^j$, $b=t^rx^ky^l$ and $c=t^sx^my^n$. Then
\begin{align*}
&\omega(\tau,\sigma)(a,b,c)=\tau_{\pi(c)}(p(a)\lhd \pi(b),p(b))\sigma_{p(a)}(\pi(b),p(b)\rhd\pi(c))\\
&=\tau_{t^s}((x^i\lhd t^r)(y^j\lhd t^r),x^ky^l)
=\tau_{t^s}(x^{{(-1)^{r}i}} y^j,x^ky^l) =(-1)^{jks}.
\end{align*}
        \item $H_{b:y}:=H_{0,1}$. Since $\sigma(t,t)=\sum_{p,q}(-1)^qe_{p,q}$,  it follows that $\sigma_{{x^iy^j}}(t^r,t^s)=(-1)^{jrs}$. Then
\begin{align*}
&\omega(\tau,\sigma)(a,b,c)=\tau_{\pi(c)}(p(a)\lhd \pi(b),p(b))\sigma_{p(a)}(\pi(b),p(b)\rhd\pi(c))\\
&=\tau_{t^s}((x^i\lhd t^r)(y^j\lhd t^r),x^ky^l) \sigma_{{x^iy^j}}(t^r,t^s)=\tau_{t^s}(x^{{(-1)^{r}i}} y^j,x^ky^l) (-1)^{jrs}\\
&=(-1)^{jks}(-1)^{jrs}=(-1)^{j(k+r)s}.
\end{align*}
\item $H_{b:x^2y}:=H_{1,1}$.
Since $\sigma(t,t)=\sum_{p,q}(-1)^{p+q}e_{p,q}$, it follows $\sigma_{{x^iy^j}}(t^r,t^s)=(-1)^{(i+j)rs}$. Then
\begin{align*}
&\omega(\tau,\sigma)(a,b,c)=\tau_{\pi(c)}(p(a)\lhd \pi(b),p(b))\sigma_{p(a)}(\pi(b),p(b)\rhd\pi(c))\\
&=\tau_{t^s}((x^i\lhd t^r)(y^j\lhd t^r),x^ky^l) \sigma_{{x^iy^j}}(t^r,t^s)=\tau_{t^s}(x^{{(-1)^{r}i}} y^j,x^ky^l) (-1)^{(i+j)rs}\\
&=(-1)^{jks}(-1)^{(i+j)rs}=(-1)^{(jk+jr+ir)s}.
\end{align*}
\end{enumerate}

\textbf{Case (c):} The action is given by $t\rightharpoonup e_{i,j}=e_{i,i+j}$ and $\sigma(t,t)=\sum_{p,q}(-1)^{{\frac{p(p-1)}{2}}}\theta^{kp}e_{p,q}$ for $\theta$ a primitive $4$-th root of unity. The Singer pair forms the Hopf algebra $H_k$ for $0\leq k\leq1$.

We claim that $F\bowtie\Gamma$ is isomorphic to
\begin{align*}
      G_6:=\langle x,t,y\mid x^4=t^2=y^2=1, [y,x]=[y,t]=1, [x,t]=y\rangle.
\end{align*}
Indeed, $x\rhd t=y\rhd t=t$, $x\lhd t=xy$ and $y\lhd t=y$, since the action $\rhd$ is trivial and
\begin{align*}
\langle x\lhd t, e_{i,j}\rangle=\langle x, t\rightharpoonup e_{i,j}\rangle=\langle x, e_{i,i+j}\rangle=\delta_{1,i}\delta_{0,i+j},\\
\langle y\lhd t, e_{i,j}\rangle=\langle y, t\rightharpoonup e_{ij}\rangle=\langle y, e_{i,i+j}\rangle=\delta_{0,i}\delta_{1,i+j}.
\end{align*}
 Then the claim follows by the following relations
\begin{align*}
(1\bowtie x) (t\bowtie1) =(x\rhd t) \bowtie (x\lhd t)=t\bowtie xy,\quad (1\bowtie y)(t\bowtie1)= t\bowtie y.
 \end{align*}
Now we calculate $\omega(\sigma,\tau)\in H^3(F\bowtie\Gamma,\K^{\times})$.
\begin{enumerate}
\item $H_{c:\sigma_0}:=H_0$. Since $\sigma(t,t)=\sum_{p,q}(-1)^{{\frac{p(p-1)}{2}}}e_{p,q}$,  it follows that $\sigma_{{x^iy^j}}(t^r,t^s)=(-1)^{{\frac{i(i-1)}{2}rs}}$. Then
\begin{align*}
&\omega(\tau,\sigma)(a,b,c)=\tau_{\pi(c)}(p(a)\lhd \pi(b),p(b))\sigma_{p(a)}(\pi(b),p(b)\rhd\pi(c))\\
&=\tau_{t^s}((x^i\lhd t^r)(y^j\lhd t^r),x^ky^l)\sigma_{{x^iy^j}}(t^r,t^s)\\
&=\tau_{t^s}(x^iy^{ri+j},x^ky^l) (-1)^{{\frac{i(i-1)}{2}rs}}\\
&=(-1)^{(ri+j)ks}(-1)^{{\frac{i(i-1)}{2}}rs}=(-1)^{{\frac{2(ri+j)ks+i(i-1)rs}{2}}}.
\end{align*}

\item $H_{c:\sigma_1}:=H_1$. Since $\sigma(t,t)=\sum_{p,q}(-1)^{{\frac{p(p-1)}{2}}}\theta^pe_{p,q}$, it follows that
$\sigma_{{x^iy^j}}(t^r,t^s)=(-1)^{{\frac{i(i-1)}{2}}rs}\theta^{irs}$. Then
 \begin{align*}
&\omega(\tau,\sigma)(a,b,c)=\tau_{\pi(c)}(p(a)\lhd \pi(b),p(b))\sigma_{p(a)}(\pi(b),p(b)\rhd\pi(c))\\
&=\tau_{t^s}((x^i\lhd t^r)(y^j\lhd t^r),x^ky^l)\sigma_{{x^iy^j}}(t^r,t^s)\\
&=\tau_{t^s}(x^iy^{ri+j},x^ky^l)(-1)^{{\frac{i(i-1)}{2}}rs}\theta^{irs}\\
&=(-1)^{(ri+j)ks}(-1)^{{\frac{i(i-1)}{2}}rs}\theta^{irs}=(-1)^{{\frac{2(ri+j)ks+i(i-1)rs}{2}}}\theta^{irs}.
\end{align*}
      \end{enumerate}

\subsection*{Case \uppercase\expandafter{\romannumeral2}: $\Gamma=\mathbb{Z}_{2}\times \mathbb{Z}_{2}\times \mathbb{Z}_{2}=\langle x\rangle\times \langle y\rangle\times\langle z\rangle$.}
Let $\{e_{p,q,r}\}_{0\leq p,q,r\leq 1}$ be the dual basis in $\K^{\Gamma}$, that is, $\langle e_{p,q,r},x^iy^jz^k\rangle=\delta_{p,i}\delta_{q,j}\delta_{r,k}$ for $0\leq i, j,k< 2$. In this case,
\begin{align*}
t\rightharpoonup e_{i,j,k}=e_{j,i,k},\quad \theta(t)=\sum_{ijkpqr}(-1)^{k(p+q)}\xi^{jp}_1 e_{i,j,k}\otimes e_{p,q,r}, \quad \sigma(t,t)=\sum\xi_1^{pq}e_{p,q,r}\sum\iota^{r}e_{p,q,r},
\end{align*}
where $\xi_1,\iota=0,1$. The Singer pair forms the  Hopf algebra $H_{d:\xi_1,\iota}$. Furthermore,
\begin{align*}
&\tau_{{t^s}}(x^{f}y^{g}z^{h},x^{l}y^{m}z^{n})=\tau(x^{f}y^{g}z^{h},x^{l}y^{m}z^{n})(t^s)
=\theta(t^s)(x^{f}y^{g}z^{h},x^{l}y^{m}z^{n})\\
&=\sum_{ijkpqr}(-1)^{ks(p+q)}\xi^{jps}_1\langle e_{i,j,k}\otimes e_{p,q,r},x^{f}y^{g}z^{h}\otimes x^{l}y^{m}z^{n}\rangle\\
&=(-1)^{sh(l+m)}\xi^{gls}_1.
\end{align*}

We claim that $F\bowtie\Gamma$ is isomorphic to
\begin{align*}
      D_8\times \mathbb{Z}_2 =G_8:=\langle xt,y,z\mid (xt)^4=y^2=z^2=1,y(xt)y=(xt)^3,[xt,z]=[y,z]=1\rangle.
\end{align*}
Indeed, $x\rhd t=y\rhd t=z\rhd t=t$, $x\lhd t=y$, $y\lhd t=x$,  since $\rhd$ is trivial and
\begin{align*}
\langle x\lhd t, e_{i,j,k}\rangle=\langle x, t\rightharpoonup e_{i,j,k}\rangle=\langle x, e_{j,i,k}\rangle=\delta_{1,j}\delta_{0,i}\delta_{0,k},\\
\langle y\lhd t, e_{i,j,k}\rangle=\langle y, t\rightharpoonup e_{i,j,k}\rangle=\langle y, e_{j,i,k}\rangle=\delta_{0,j}\delta_{1,i}\delta_{0,k}.
\end{align*}
Similarly, we have $z\lhd t=z$. Then the claim follows by the following relations
\begin{gather*}
(1\bowtie x)(t\bowtie 1)=(x\rhd t)\bowtie (x\lhd t)=t\bowtie y,\quad (1\bowtie y)(t\bowtie1)=t\bowtie x,\\(1\bowtie z)(t\bowtie1)=(z\rhd t)\bowtie(z\lhd t)=t\bowtie z.
\end{gather*}

Now we calculate $\omega(\sigma,\tau)\in Z^3(F\bowtie\Gamma,\K^{\times})$.
\begin{enumerate}
\item $H_{d:1,1}$. Then $\sigma(t,t)=\sum e_{p,q,r}\sum e_{p,q,r}$ is trivial, that is, $\sigma_{x^iy^jz^l}(t^r,t^s)=1$.
Let $a=t^px^iy^jz^k$, $b=t^qx^ly^mz^n$ and $c=t^rx^fy^gz^h$. Then
\begin{align*}
&\omega(\tau,\sigma)(a,b,c)=\tau_{\pi(c)}(p(a)\lhd \pi(b),p(b))\sigma_{p(a)}(\pi(b),p(b)\rhd\pi(c))\\
&=\tau_{t^r}((x^i\lhd t^q)(y^j\lhd t^q)z^k,x^ly^mz^n)\\
&=\tau_{t^r}(x^{{\frac{(i+j)+(-1)^{q}(i-j)}{2}}}y^{{\frac{(i+j)+(-1)^{q}(j-i)}{2}}}z^k,x^ly^mz^n)\\
&=(-1)^{kr(l+m)}.
\end{align*}
        \item $H_{d:1,-1}$. Since  $\sigma(t,t)=\sum e_{p,q,r}\sum(-1)^{r}e_{p,q,r}$,
it follows that  $\sigma_{x^iy^jz^l}(t^r,t^s)=(-1)^{rsl}$. Then
\begin{align*}
&\omega(\tau,\sigma)(a,b,c)=\tau_{\pi(c)}(p(a)\lhd \pi(b),p(b))\sigma_{p(a)}(\pi(b),p(b)\rhd\pi(c))\\
&=\tau_{t^r}((x^i\lhd t^q)(y^j\lhd t^q)z^k,x^ly^mz^n)\sigma_{x^iy^jz^l}(t^q,t^r)\\
&=\tau_{t^r}(x^{{\frac{(i+j)+(-1)^{q}(i-j)}{2}}}y^{{\frac{(i+j)+(-1)^{q}(j-i)}{2}}}z^k,
x^ly^mz^n)\sigma_{x^iy^jz^l}(t^q,t^r)\\
&=(-1)^{kr(l+m)}(-1)^{qrl}=(-1)^{(kl+km+ql)r}.
\end{align*}
        \item $H_{d:-1,1}$. Since  $\sigma(t,t)=\sum(-1)^{pq}e_{p,q,r}\sum e_{p,q,r}$, it follows that
   $\sigma_{x^iy^jz^l}(t^r,t^s)=(-1)^{rsij}$. Then
\begin{align*}
&\omega(\tau,\sigma)(a,b,c)=\tau_{\pi(c)}(p(a)\lhd \pi(b),p(b))\sigma_{p(a)}(\pi(b),p(b)\rhd\pi(c))\\
&=\tau_{t^r}((x^i\lhd t^q)(y^j\lhd t^q)z^k,x^ly^mz^n)\sigma_{x^iy^jz^l}(t^q,t^r)\\
&=\tau_{t^r}(x^{{\frac{(i+j)+(-1)^{q}(i-j)}{2}}}y^{{\frac{(i+j)+(-1)^{q}(j-i)}{2}}}z^k,x^ly^mz^n)\sigma_{x^iy^jz^l}(t^q,t^r)\\
&=(-1)^{kr(l+m)+\frac{(i+j)+(-1)^{q}(j-i)}{2}lr}(-1)^{qrij}.
\end{align*}
        \item $H_{d:-1,-1}$. Since  $\sigma(t,t)=\sum(-1)^{pq}e_{p,q,r}\sum(-1)^{r}e_{p,q,r}$, it follows that
  $\sigma_{x^iy^jz^l}(t^r,t^s)=(-1)^{rs(ij+l)}$. Then
\begin{align*}
&\omega(\tau,\sigma)(a,b,c)=\tau_{\pi(c)}(p(a)\lhd \pi(b),p(b))\sigma_{p(a)}(\pi(b),p(b)\rhd\pi(c))\\
&=\tau_{t^r}((x^i\lhd t^q)(y^j\lhd t^q)z^k,x^ly^mz^n)\sigma_{x^iy^jz^l}(t^q,t^r)\\
&=\tau_{t^r}(x^{{\frac{(i+j)+(-1)^{q}(i-j)}{2}}}y^{{\frac{(i+j)+(-1)^{q}(j-i)}{2}}}z^k,x^ly^mz^n) \sigma_{x^iy^jz^l}(t^q,t^r)\\
&=(-1)^{kr(l+m)+\frac{(i+j)+(-1)^{q}(j-i)}{2}lr}(-1)^{qr(ij+l)}.
\end{align*}
   \end{enumerate}

\subsection*{Case \uppercase\expandafter{\romannumeral3}: $\Gamma=D_{8}=\langle x, y\mid x^{4}=y^{2}=1,yx=x^{-1}y\rangle$.}
Let $\{e_{p,q}\}_{\substack{0\leq p\leq3,\,0\leq q\leq1}}$ be the dual basis in $\K^{\Gamma}$, that is, $\langle e_{p,q},x^{i}y^{j} \rangle=\delta_{p,i}\delta_{q, j}$. Then
$X=\sum_{pq}(-1)^{p}e_{p,q}$, $Y=\sum_{pq}(-1)^{q}e_{p,q}$ are group-like elements of order $2$ and $\theta(t)=\sum_{ijpq}\xi^{jp}e_{ij}\otimes e_{pq}$, where $\xi^4=1$. Furthermore, we have
 $$\tau_{{t^s}}(x^iy^j,x^ky^l)=\xi^{{jks}}.$$

 \textbf{Case (B)}: The action is given by $t\rightharpoonup e_{p,q}=e_{-p,q}$ and $\sigma(t,t)=\sum(-1)^{kp+lq}e_{p,q}=X^{k}Y^{l}, 0\leq k,l \leq 1$.  The Singer pair forms the Hopf algebra  $H_{X^kY^l}$. Moreover,   $\xi$ is a primitive $4$-th root of unity.

We claim that $F\bowtie\Gamma$ is isomorphic to
\begin{align*}
     D_8\times\mathbb{Z}_2=G_8:=\langle x,t,yt\mid x^4=t^2=(yt)^2=1,tx=x^{-1}t,[x,yt]=[t,y]=1\rangle.
\end{align*}
Indeed, $x\rhd t=y\rhd t=t$, $x\lhd t=x^{-1}$ and $y\lhd t=y$, since $\rhd$ is trivial and
\begin{align*}
\langle x\lhd t, e_{i,j}\rangle=\langle x, t\rightharpoonup e_{i,j}\rangle=\langle x, e_{-i,j}\rangle=\delta_{1,-i}\delta_{0,j},\\
\langle y\lhd t, e_{i,j}\rangle=\langle y, t\rightharpoonup e_{i,j}\rangle=\langle y, e_{-i,j}\rangle=\delta_{0,-i}\delta_{1,j}.
\end{align*}
Then the claim follows by the following relations
\begin{align*}
(1\bowtie x)(t\bowtie1)=t\bowtie x^{-1},\quad(1\bowtie y)(t\bowtie1)=t\bowtie y.
\end{align*}

Now we calculate $\omega(\sigma,\tau)\in Z^3(F\bowtie\Gamma,\K^{\times})$.
\begin{enumerate}
\item $H_{B:1}:=H_{X^0Y^0}$. Then $\sigma(t,t)=\sum e_{p,q}$ is trivial, that is, $\sigma_{{x^iy^j}}(t^r,t^s)=1$. Let $a=t^qx^iy^j,b=t^rx^ky^l,c=t^sx^my^n$ for some $0\leq q,j,r,k,s$, $n<2,0\leq i,k,m<4$. Then
\begin{align*}
        &\omega(\tau,\sigma)(a,b,c)=\tau_{\pi(c)}(p(a)\lhd \pi(b),p(b)) \sigma_{p(a)} (\pi(b),p(b)\rhd\pi(c))\\
        &=\tau_{t^s}((x^i\lhd t^r)y^j,x^ky^l)=\tau_{t^s}(x^{{(-1)^{r}i}}y^j,x^ky^l)=\xi^{jks}.
\end{align*}
\item $H_{B:X}:=H_{X^1Y^0}$. Since $\sigma(t,t)=\sum (-1)^p e_{p,q}$, it follows that $\sigma_{{x^iy^j}}(t^r,t^s)=(-1)^{irs}$. Then
\begin{align*}
&\omega(\tau,\sigma)(a,b,c)=\tau_{\pi(c)}(p(a)\lhd \pi(b),p(b)) \sigma_{p(a)} (\pi(b),p(b)\rhd\pi(c))\\
&=\tau_{{t^s}}((x^i\lhd t^r) y^j,x^ky^l)\sigma_{{x^iy^j}}(t^r,t^s)=\tau_{t^s}(x^{(-1)^{r}i}y^j,x^ky^l)(-1)^{irs}\\
&=(-1)^{irs}\xi^{jks}.
\end{align*}
\end{enumerate}

\textbf{Case (C)}: The action is given by $t\rightharpoonup e_{p,q}=e_{-p+q,q}$ and $\sigma_k(t,t)=\sum\mu^{kq}e_{p,q}$ for $k=0,1$, where   $\xi=\mu^{-2k}$ and   $\mu$ is a primitive $8$-th root of unity. The Singer pair forms the Hopf algebra  $H_k$.

We claim that $F\bowtie\Gamma$ is isomorphic to
\begin{align*}
     D_{16}=G_3:=\langle ty,y\mid (ty)^8=y^2=1,(ty)y=y(ty)^{-1}\rangle.
\end{align*}
Indeed, $x\rhd t=y\rhd t=t$, $x\lhd t=x^{-1}$ and $y\lhd t=xy$, since $\rhd$ is trivial and
\begin{align*}
\langle x\lhd t, e_{i,j}\rangle=\langle x, t\rightharpoonup e_{i,j}\rangle=\langle x, e_{-i+j,j}\rangle=\delta_{1,-i+j}\delta_{0,j},\\
\langle y\lhd t, e_{i,j}\rangle=\langle y, t\rightharpoonup e_{i,j}\rangle=\langle y, e_{-i+j,j}\rangle=\delta_{0,-i+j}\delta_{1,j}.
\end{align*}
Then the claim follows by the following relations
\begin{gather*}
(1\bowtie x)(t\bowtie1)=(x\rhd t)\bowtie(x\lhd t)=t\bowtie x^{-1},\\
(1\bowtie y)(t\bowtie1)=(y\rhd t)\bowtie(y\lhd t)=t\bowtie xy.
\end{gather*}

Now we calculate $\omega(\sigma,\tau)\in Z^3(F\bowtie\Gamma,\K^{\times})$.
\begin{enumerate}
        \item $H_{C:1}:=H_0$.
Since $\sigma(t,t)=\sum e_{p,q}$ and $\xi=1$, it follows that $\sigma$ and $\tau$ are trivial and  hence $\omega(\tau,\sigma)$ is trivial.

\item $H_{C:\sigma_1}:=H_1$. Since $\sigma(t,t)=\sum \omega^q e_{p,q}$, it follows that $\sigma_{{x^iy^j}}(t^r,t^s)=\mu^{jrs}$. Then
\begin{align*}
\omega(\tau,\sigma)(a,b,c)=\xi^{jks}\mu^{jrs}=\mu^{(r-2k)js}.
\end{align*}
\end{enumerate}

\subsection*{Case \uppercase\expandafter{\romannumeral4}:
$\Gamma=Q_{8}=\langle x, y\mid x^{4}=1,y^{2}=x^2,yx=x^{-1}y\rangle$.}
Let $\{e_{p,q}\}_{\substack{0\leq p\leq3,\,0\leq q\leq1}}$ be the  basis of $\K^{\Gamma}$ dual to the basis $\{x^py^q\}$. Then the action is given by $t\rightharpoonup e_{i,j}=e_{-i+j,j}$, $\sigma$ and $\tau$ are trivial and  hence the $3$-cocycle $\omega(\tau,\sigma)$ is trivial.  The Singer pair forms the Hopf algebra $H_E$.

Recall that $x\rhd t=y\rhd t=t$. Since $(1\bowtie x)(t\bowtie 1)=(x\rhd t)\bowtie(x\lhd t)=t\bowtie x^{-1}$, $(1\bowtie y)(t\bowtie 1)=(y\rhd t)\bowtie(y\lhd t)=t\bowtie xy$, it follows that $F\bowtie\Gamma$ is isomorphic to
\begin{align*}
  SD_{16}=G_2:=\langle ty,t\mid (ty)^8=t^2=1,t(ty)t=(ty)^3\rangle.
\end{align*}

\subsection{Isomorphism classes of Galois objects}
In this subsection, we determine isomorphism classes of Galois objects of the Hopf algebras listed in \cite[Table 1]{Ka}. We first introduce some necessary materials that will be used later.

\begin{lem}\label{subgroups8}Let $L$ be a subgroup of order $8$ of  $F\bowtie\Gamma$ such that $L\cap F=1$, where $F\bowtie\Gamma$ is one of the groups of order $16$ mentioned in subsection $\ref{subsectionMatchedpair}$. Then $L$ is isomorphic to one of the following groups:
\begin{description}
 \item[Case \uppercase\expandafter{\romannumeral1} (a)]
  $\Z_4\times \Z_2:\langle x\rangle\times\langle y\rangle,\  \langle x\rangle\times \langle tx^3y\rangle;\quad D_8:\langle xt,y\rangle;\quad Q_8:\langle xt,yt\rangle$.
  \item[Case \uppercase\expandafter{\romannumeral1} (b)]
    $D_8:\langle x,ty\rangle, \langle xy,ty\rangle$;
  \quad $\Z_4\times \Z_2:\langle x\rangle\times\langle y\rangle;$\quad $\Z_2\times \Z_2\times \Z_2:\langle x^2\rangle\times \langle y\rangle\times \langle tx\rangle$.
  \item[Case \uppercase\expandafter{\romannumeral1} (c)]
  $\mathbb{Z}_4\times \mathbb{Z}_2:\langle x\rangle\times\langle y\rangle,\  \langle tx\rangle\times\langle y\rangle$.
  \item[Case \uppercase\expandafter{\romannumeral2}]
  $\mathbb{Z}_4\times \mathbb{Z}_2:\langle xt\rangle\times\langle z\rangle;\quad \mathbb{Z}_2\times \mathbb{Z}_2\times \mathbb{Z}_2:\langle x\rangle\times \langle y\rangle\times \langle z\rangle; D_8:\langle xt,yz\rangle, \;\langle (xt)z,y\rangle$.
  \item[Case \uppercase\expandafter{\romannumeral3} (B)]
  $\mathbb{Z}_4\times \mathbb{Z}_2:\langle x\rangle\times\langle yt\rangle;\quad \mathbb{Z}_2\times \mathbb{Z}_2\times \mathbb{Z}_2:\langle x^2\rangle\times \langle xt\rangle\times \langle yt\rangle;\\ D_8:\langle x,t(yt)\rangle=\langle x,y\rangle,\;\langle x(yt),t(yt)\rangle=\langle x(yt),y\rangle$.
  \item[Case \uppercase\expandafter{\romannumeral3} (C)]
  $\mathbb{Z}_8:\langle ty\rangle$;\quad $D_8:\langle x,y\rangle$.
  \item[Case \uppercase\expandafter{\romannumeral4}]
  $\mathbb{Z}_8:\langle ty\rangle$;\quad $Q_8:\langle x,y\rangle$.
\end{description}
\end{lem}
\begin{proof}
%It follows by direct computations.
For Case \uppercase\expandafter{\romannumeral1} (a),  $F\bowtie\Gamma\cong G_7$ with generators $x,y,t$. Then   there are seven maximal subgroups, that is,
\begin{align*}
\langle x\rangle\times\langle t\rangle,\quad \langle y,t\rangle,\quad \langle yx,t\rangle,\quad \langle x\rangle\times \langle y\rangle,\quad \langle x\rangle\times\langle ytx\rangle=\langle x\rangle\times\langle tx^3y\rangle,\quad \langle xt,y\rangle,\quad \langle xt,yt\rangle.
\end{align*}

For Case \uppercase\expandafter{\romannumeral1} (b), $F\bowtie\Gamma\cong D_8\times \Z_2$ with generators $x,t,y$. Then there are seven maximal subgroups, that is,
\begin{align*}
\langle x\rangle\times\langle y\rangle,\quad  \langle x^2\rangle\times\langle y\rangle\times\langle t\rangle,\quad \langle x^2\rangle\times\langle y\rangle\times\langle tx\rangle,\quad \langle x,t\rangle,\quad \langle x,ty\rangle,\quad \langle xy,t\rangle,\quad \langle xy,ty\rangle.
\end{align*}

For Case \uppercase\expandafter{\romannumeral1} (c), $F\bowtie\Gamma\cong G_6$ with generators $x,t,y$. Then there are three maximal subgroups, that is,
\begin{align*}
\langle x\rangle\times\langle y\rangle,\quad \langle x^2\rangle\times\langle y\rangle\times\langle t\rangle,\quad \langle tx\rangle\times\langle y\rangle.
\end{align*}

For Case \uppercase\expandafter{\romannumeral2}, $F\bowtie\Gamma\cong D_8\times \Z_2$ with generators $xt,y,z$. Then there are seven maximal subgroups, that is,
\begin{align*}
\langle xt\rangle\times\langle z\rangle,\quad  \langle (xt)^2\rangle\times\langle z\rangle\times\langle y\rangle,\quad \langle (xt)^2\rangle\times\langle y(xt)\rangle\times\langle z\rangle,\quad \langle xt,y\rangle,\quad \langle xt,yz\rangle,\quad \langle (xt)z,y\rangle,\quad \langle (xt)z,yz\rangle.
\end{align*}
Observe that $(xt)^2=xy,(xt)^3=tx=yt$ and $(xt)zyz=xty=t$.

For Case \uppercase\expandafter{\romannumeral3} (B), $F\bowtie\Gamma\cong D_8\times \Z_2$ with generators $x,t,yt$. Then there are seven maximal subgroups, that is,
\begin{align*}
\langle x\rangle\times\langle yt\rangle,\quad  \langle x^2\rangle\times\langle yt\rangle\times\langle t\rangle,\quad \langle x^2\rangle\times\langle yt\rangle\times\langle xt\rangle,\quad \langle x,t\rangle,\quad \langle x,t(yt)\rangle,\quad \langle x(yt),t\rangle,\quad \langle x(yt),t(yt)\rangle.
\end{align*}

For Case \uppercase\expandafter{\romannumeral3} (C), $F\bowtie\Gamma\cong D_{16}$ with generators $ty,y$. Then there are three maximal subgroups, that is,
\begin{align*}
\langle ty\rangle,\quad  \langle x,y\rangle,\quad \langle x,tx\rangle.
\end{align*}

For  Case \uppercase\expandafter{\romannumeral4}, $F\bowtie\Gamma\cong SD_{16}$ with generators $ty,t$. Then there are three maximal subgroups, that is,
\begin{align*}
\langle ty\rangle,\quad  \langle x,t\rangle,\quad \langle x,y\rangle.
\end{align*}
Consequently, the statements follow by direct computations.
\end{proof}

 The Schur multipliers of groups of order $8$ are given by the following proposition.
\begin{pro} [\cite{Kar}]\label{Schurmultiplier}
Let $L$ be a group of order $8$. Then
 the Schur multiplier   $H^2(L,\K^{\times})$  is given as follow: $H^2(\mathbb{Z}_8,\K^{\times})=H^2(Q_8,\K^{\times})=0$, $H^2(\mathbb{Z}_4\times\mathbb{Z}_2,\K^{\times})\cong H^2(D_8,\K^{\times})\cong\mathbb{Z}_2$,
$H^2(\mathbb{Z}_2\times\mathbb{Z}_2\times\mathbb{Z}_2,\K^{\times})\cong
\mathbb{Z}_2\times\mathbb{Z}_2\times\mathbb{Z}_2$.
\end{pro}

\begin{rmk}\label{2cocycleabelian}Let $G$ be a finite abelian group and $\sigma\in Z^2(G,\K^{\times})$. Then the class of  $\sigma$ is trivial, if and only if $\K_\sigma[G]$ is commutative, if and only if $\sigma$ is symmetric, that is, $\sigma(g,h)=\sigma(h,g)$ for any $g,h\in G$. Indeed, if the class of $\sigma$ is trivial, then $\K_\sigma[G]\cong \K[G]$ is commutative; if $\K_\sigma[G]$ is commutative, then irreducible representations of $\K_\sigma[G]$ are one-dimensional, hence there exists a morphism $\varphi:G\to \K^{\times}$ such that $\varphi(g._\sigma h)=\varphi(g)\varphi(h)$, that is, $\varphi(gh)^{-1}\varphi(g)\varphi(h)=\sigma(g,h)$. Obviously, $\K_\sigma[G]$ is commutative if and only if $\sigma$ is symmetric.
\end{rmk}

Now we describe explicitly $H^2(\mathbb{Z}_2\times\mathbb{Z}_2\times\mathbb{Z}_2,\K^{\times})$, which was determined in  \cite{BPW}.
\begin{pro}$($\cite[Proposition 4.10]{BPW}$)$\label{cohomology} A set of cocycles $\sigma:\mathbb{Z}_2\times\mathbb{Z}_2\times\mathbb{Z}_2\to \K^{\times}$ which represents $H^2(\mathbb{Z}_2\times\mathbb{Z}_2\times\mathbb{Z}_2,\K^{\times})$ is given by
$$\sigma(g^{i_1}_1g^{i_2}_2g^{i_3}_3,g^{j_1}_1g^{j_2}_2g^{j_3}_3)
=(-1)^{a_{12}i_1j_2+a_{13}i_1j_3+a_{23}i_2j_3},$$
where $g_1,g_2,g_3$ are generators of $\mathbb{Z}_2\times\mathbb{Z}_2\times\mathbb{Z}_2$ and $i_l,j_k,a_{mn}\in{\{0,1}\}$ for $1\leq k,l,m,n\leq 3$.
\end{pro}
Let $\alpha\in H^2(\mathbb{Z}_2\times\mathbb{Z}_2\times\mathbb{Z}_2,\K^{\times})$. Then  by Proposition \ref{cohomology},   $\alpha$ is uniquely determined by a sequence $(a_{12},a_{13},a_{23})$. Hence $(a_{12},a_{13},a_{23})$ can be used to label $\alpha$.

Let $G\cong \Z_{m_{1}}\times\cdots \times\mathbb{Z}_{m_{n}}:=\langle g_1\rangle\times\cdots\times\langle g_n\rangle$. An explicit and unified formula of the normalized 3-cocycle on  $G$ is determined  in \cite{HLYY}.
Let $A$   be the set of all sequences
\begin{equation}\label{3.1}
\underline{\mathbf{a}}:=(a_{1},\ldots,a_{l},\ldots,a_{n},a_{12},\ldots,a_{ij},\ldots,a_{n-1,n},a_{123},\ldots,a_{rst},\ldots,a_{n-2,n-1,n})
\end{equation}
such that $ 0\leq a_{l}<m_{l}, \ 0\leq a_{ij}<(m_{i},m_{j}), \ 0\leq a_{rst}<(m_{r},m_{s},m_{t})$ for $1\leq l\leq n, \ 1\leq i<j\leq n, \ 1\leq r<s<t\leq n$ where $a_{ij}$ and $a_{rst}$ are ordered by the lexicographic order. Here  $(s_{1},\ldots,s_{t})$  denotes the greatest common divisor of   $s_{1},\ldots,s_{t}\in\Z_{+}$.

Let $\nu_{n}$ be a primitive $n$-th root of unity for any positive integer $n$,  and $[\frac{r}{n}]$    the
largest integer less than $\frac{r}{n}$ for $r \in\Z$. For any $\underline{\mathbf{a}}\in A$, define a $\mathbb{Z}G$-module morphism:
\begin{eqnarray}\label{3cocycle}
&& \omega_{\underline{\mathbf{a}}}: G \times G\times G \rightarrow\K^{\times} \notag \\
&&[g_{1}^{i_{1}}\cdots g_{n}^{i_{n}},g_{1}^{j_{1}}\cdots g_{n}^{j_{n}},g_{1}^{k_{1}}\cdots g_{n}^{k_{n}}] \mapsto \prod_{l=1}^{n}\nu_{m_{l}}^{a_{l}i_{l}[\frac{j_{l}+k_{l}}{m_{l}}]}
\prod_{1\leq s<t\leq n}\nu_{m_{t}}^{a_{st}i_{t}[\frac{j_{s}+k_{s}}{m_{s}}]}
\prod_{1\leq r<s<t\leq n}\nu_{(m_{r},m_{s},m_{t})}^{-a_{rst}k_{r}j_{s}i_{t}}. \notag
\end{eqnarray}

\begin{pro}\cite[Proposition 3.8]{HLYY}
$\{\omega_{\ULa}|\ULa\in A\}$ is a complete set of representatives of normalized $3$-cocycles on $G.$
\end{pro}

\begin{rmk}\label{cocyclerepresent}
If $G\cong\Z_N:=\langle g\rangle$, then a complete set of representative of $H^3(G,\K^\times)$ is given by
$$
\omega_a(g^i,g^j,g^k)=\nu_N^{ak[\frac{i+j}{N}]},\quad 0\leq a, i,j,k< N.
$$

If $G\cong\Z_N\times\Z_M:=\langle g\rangle\times\langle h\rangle$, then a complete set of representative of $H^3(G,\K^\times)$ is given by
\begin{align*}
\omega_{\ULa}(g^ih^j,g^kh^l,g^mh^n)
=\nu_M^{a_1i[\frac{k+m}{M}]}\nu_N^{a_2j[\frac{l+n}{N}]}\nu_{N}^{a_{12}j[\frac{k+m}{M}]},\quad 0\leq i,k,m<M,0\leq j,l,n<N,
\end{align*}
where $\ULa:=(a_1,a_2,a_{12})\in A$.
\end{rmk}

Let $H=(\K^{\Gamma}{}^{\rho,\theta}\#_{\rightharpoonup,\sigma}\K[F])^*$   be a Hopf algebra listed in \cite[Table 1]{Ka}. Now we compute the number of Galois objects of $H$ up to isomorphism.  By Proposition \ref{Galoisfiber}, we proceed this by computing the number  of isomorphism classes of fiber functors $(L,\alpha)$ on  $\mathcal{C}(F\bowtie\Gamma,\omega(\tau,\sigma),F,1)$.

From Theorem \ref{fiberfuncondition},    condition $(3)$  forces $L\cap F=1$; then condition $(2)$ implies  $L$ is a normal subgroup of order $8$ and $F\bowtie\Gamma = LF$ is an exact factorization; furthermore,   fiber functors $(L,\alpha)$, $(L{'},\alpha{'})$ are isomorphic  if and only if $L=L'$ and there exists an element $g\in G$ such that the class $\alpha^{-1}\alpha'^g\Omega_g$ is trivial in $H^2(L, \K^{\times})$.

Consequently,  it suffices to find all subgroups $L$ of order $8$ satisfying the conditions
\begin{align}
\label{conditionC}  \text{$L\cap F=1$ and  the class of $\omega(\tau,\sigma)|_{ L\times L\times L}$ is trivial},
\end{align}
 and then determine all isomorphism classes of pairs $(L,\alpha)$ for $[\alpha]\in H^2(L,\K^\times)$.  In what follows, unless specified otherwise, $L$ denotes a subgroup of order $8$ satisfying the condition \eqref{conditionC} and $\omega:=\omega(\tau,\sigma)$.

\begin{rmk}\label{rmk-1}
If $L\cong\Z_4\times\Z_2$ or $D_8$, then by Proposition \ref{Schurmultiplier}, $H^2(L,\K^{\times})\cong \Z_2$. Let $[\alpha]$ be a generator of $H^2(L,\K^{\times})$. Then by Theorem \ref{fiberfuncondition}, $(L,1)$ and $(L,\alpha)$ are isomorphic if  and only if
\begin{itemize}
\item the class of $\alpha\Omega_g|_{L\times L}$ is trivial in $H^2(L,\K^{\times})$ for some $g\in G$.
\end{itemize}
Furthermore, if $L\cong\Z_4\times\Z_2$,  by Remark \ref{2cocycleabelian}, they are isomorphic if and only if $\Omega_g|_{L\times L}$ is not symmetric.
\end{rmk}

\begin{lem}\label{H(a,y)}
Up to isomorphism, $(H_{a:1})^*$ has one or two Galois objects, and $(H_{a:y})^*$ has four Galois objects.
\end{lem}
\begin{proof}
By Lemma \ref{subgroups8}, the possibilities of $L$ are $\langle x\rangle\times\langle y\rangle\cong\Z_4\times \Z_2$, $\langle x\rangle\times \langle tx^3y\rangle\cong\Z_4\times \Z_2$, $\langle xt,y\rangle\cong D_8$,  $\langle xt,yt\rangle\cong Q_8$. Let $a=t^qx^iy^j$,  $b=t^rx^ky^l$ and $c=t^sx^my^n$ for some $q,r,s,j,l,n\in\Z_2,i,k,m\in\Z_4$.

For $H_{a:1}$,  we have $\omega(a,b,c)=(-1)^{jks}$. We first claim that $L\neq \langle x\rangle\times \langle tx^3y\rangle$ or $\langle xt,y\rangle$. Indeed, consider the restriction of $\omega$ on the subgroup $\langle tx^3y\rangle\cong\Z_2$ of $L$, we have
$$
\omega((tx^3y)^i,(tx^3y)^j,(tx^3y)^k)=(-1)^{ijk}=(-1)^{k[\frac{i+j}{2}]},\quad  i,j,k\in\{0,1\};
$$
then by Remark \ref{cocyclerepresent}, the class of $\omega|_{\langle tx^3y\rangle\times\langle tx^3y\rangle\times\langle tx^3y\rangle}$ is not trivial in $H^3(\Z_2,\K^{\times})$, which implies that the class $\omega|_{L\times L\times L}$ is non-trivial and hence the claim follows.

Assume that $L=\langle x\rangle\times\langle y\rangle\cong\Z_4\times\Z_2$. Clearly, $\omega|_{L\times L\times L}=1$. By Proposition \ref{2cocycleabelian}, $H^2(\Z_4\times\Z_2,\K^\times)\cong\Z_2$. Let $\{[\alpha]\}$ be a   set of representative of $H^2(\Z_4\times\Z_2,\K^\times)$.   We claim that $(L,1)$ and $(L,\alpha)$ are isomorphic.   By definition,
\begin{align*}
\Omega_t(x^ky^l,x^my^n)=\frac{\omega(tx^ky^lt,tx^my^nt,t)\omega(t,x^ky^l,x^my^n)}{\omega(tx^ky^lt,t,x^my^n)}=(-1)^{lm}.
\end{align*}
It is easy to see that the class of $\Omega_t|_{L\times L}$ is not symmetric. Then the claim  follows by Remark \ref{rmk-1}.

Assume that $L=\langle xt,yt\rangle\cong Q_8$. Then  by Proposition \ref{Schurmultiplier}, $H^2(Q_8,\K^\times)=0$. Hence there exists at most one Galois object for $L=\langle xt,yt\rangle$, depending on whether the class of $\omega|_{L\times L\times L}$ is trivial or not.
Consequently,  up to isomorphism there exists one or two  trivial Galois objects for $(H_{a:1})^*$.

For $H_{a:y}$, we have $\omega(a,b,c)=(-1)^{j(k+r)s}$.
Assume that $L=\langle x\rangle\times \langle tx^3y\rangle\cong\Z_4\times \Z_2$.  We claim that  $\omega|_{L\times L\times L}$ is trivial in $H^3(L,\K^{\times})$. Observe that $x^i(tx^3y)^j=t^jx^{3j+i}y^j$. Let $\beta:=tx^3y$ for short.   By Remark \ref{cocyclerepresent}, there is a 2-chain $\varphi:L\times L\to \K^\times$ and a sequence $\ULa:=(a_1,a_2,a_{12})$ such that
\begin{align*}
d^2(\varphi)\omega|_{L\times L\times L}(x^i\beta^j,x^k\beta^l,x^m\beta^n)=\omega_{\ULa}(x^i\beta^j,x^k\beta^l,x^m\beta^n)=\nu_4^{a_1i[\frac{k+m}{4}]}\nu_2^{a_2j[\frac{l+n}{2}]}\nu_2^{a_{12}j[\frac{k+m}{4}]}.
\end{align*}
Since $\omega|_{P\times P\times P}=1$ for $P=\langle x\rangle$ and $\langle \beta\rangle$, it follows that $a_1=0=a_2$. Now consider the subgroup $Q:=\langle txy\rangle=\langle x^2\beta\rangle$ of $L$, clearly $\omega|_{Q\times Q\times Q}=1$. Then
\begin{align*}
d^2(\varphi)(txy,txy,txy)=\omega_{\ULa}(txy,txy,txy)=\nu_2^{a_{12}}=(-1)^{a_{12}}.
\end{align*}
Meanwhile, $d^2(\varphi)(txy,txy,txy)=\frac{\varphi(txy,1)}{\varphi(1,txy)}$, and $\varphi(txy,1)=\varphi(1,txy)=\varphi(1,1)$, which implies $a_{12}=0$ and hence the claim follows.

Now we show that $(L,1)$ and $(L,\alpha)$ are isomorphic. Observe that $xt=tx$ and $y^lt=tx^{2l}y^l$ for $0\leq l\leq 1$. By definition,
\begin{align*}
\Omega_t(x^k\beta^l,x^m\beta^n)&=\frac{\omega(t^{l+1}x^{k+3l}y^lt,t^{n+1}x^{m+3n}y^nt,t)\omega(t,t^{l}x^{k+3l}y^l,t^nx^{3n+m}y^n)}{\omega(t^{l+1}x^{k+3l}y^lt,t,t^nx^{3n+m}y^n)}\\
&=\frac{\omega(t^{l}x^{k+l}y^l,t^{n}x^{m+n}y^n,t)\omega(t,t^{l}x^{k+3l}y^l,t^nx^{3n+m}y^n)}{\omega(t^{l}x^{k+l}y^l,t,t^nx^{3n+m}y^n)}=(-1)^{(n+m)l}.
\end{align*}
It is easy to see that $\Omega_t$ is not symmetric. Then the claim  follows by Remark \ref{rmk-1}.

Assume that $L=\langle x\rangle\times\langle y\rangle\cong\Z_4\times\Z_2$, $\langle xt,y\rangle\cong D_8$ or $\langle xt,yt\rangle\cong Q_8$.  It is easy to see that  $\omega|_{L\times L\times L}=1$.   Now we determine isomorphism classes of fiber functors. If $L=\langle x\rangle\times\langle y\rangle$, then $\Omega_t(x^ky^l,x^my^n)=(-1)^{lm}$, which is not symmetric, and hence there is only one isomorphism class.  If $L=\langle xt,yt\rangle\cong Q_8$, then by Proposition \ref{Schurmultiplier}, there is only one isomorphism class. If $L=\langle xt,y\rangle\cong D_8$, then
\begin{align*}
\Omega_t((xt)^ky^l,(xt)^my^n)=\frac{\omega(t^{k}x^{k+2l}y^l,t^{m}x^{m+2n}y^n,t)\omega(t,t^kx^ky^l,t^mx^my^n)}{\omega(t^{k}x^{k+2l}y^l,t,t^mx^my^n)}
=(-1)^{lm}
\end{align*}
which implies that the class of $\Omega_t|_{L\times L}$ is not trivial in $H^2(L,\K^{\times})$, and hence by Remark \ref{rmk-1} there is only one isomorphism. Indeed if there exists a morphism $f:L\to\K^\times$ such that $d(f)=\Omega_t|_{L\times L}$, then $f((xt)^k)=f(xt)^k$ and $f((xt)^4)=f(y^2)=f(y)^2=1$. Since $\Omega_t(xt,y)=1$, $f(y(xt))=f((xt)^3y)=f(xt)^3 f(y)$, while then $-1=\Omega_t(y,xt)=f(y)f(xt)f(y(xt))^{-1}=f(y)^2f(xt)^4=1$, impossible.

Therefore, up to isomorphism there are  four Galois objects for $(H_{a:y})^*$.
\end{proof}
\begin{rmk}
 From the proof of Lemma \ref{H(a,y)}, if the class of $\omega|_{L\times L\times L}$ for $L=\langle xt,yt\rangle\cong Q_8$ is trivial, then $(H_{a:1})^*$ has two Galois objects; otherwise $(H_{a:1})^*$ has only one Galois object.
\end{rmk}

\begin{lem}\label{H(b)}
Up to isomorphism,  $(H_{b:1})^*$ has  one or two  Galois objects, $(H_{b:y})^*$ and  $(H_{b:x^2y})^*$ have seven Galois objects.
\end{lem}
\begin{proof}
By Lemma \ref{subgroups8}, the possibilities of $L$ are $\langle x,ty\rangle\cong D_8$, $\langle xy,ty\rangle\cong D_8$,
   $ \langle x\rangle\times\langle y\rangle\cong \Z_4\times \Z_2$,   $\langle x^2\rangle\times \langle y\rangle\times \langle tx\rangle\cong \Z_2\times \Z_2\times \Z_2$.
   Let $a=t^qx^iy^j$,  $b=t^rx^ky^l$ and $c=t^sx^my^n$ for some $q,r,s,j,l,n\in\Z_2,i,k,m\in\Z_4$.

For $H_{b:1}$, we have $\omega(a,b,c)=(-1)^{jks}$.
If $L= \langle x,ty\rangle$ or $\langle x^2\rangle\times \langle y\rangle\times \langle tx\rangle$, then the restriction of $\omega$ on the subgroup $\langle tx^3y\rangle$ is not trivial in $H^3(\langle tx^3y\rangle,\K^{\times})$ and hence $\omega|_{L\times L\times L}$ is not trivial in $H^3(L,\K^{\times})$.

Assume that $L=\langle xy,ty\rangle\cong D_8$. Then a direct computation shows that the class of $\Omega_t|_{L\times L}$ is non-trivial.  If the class of $\omega|_{L\times L\times L}$ is non-trivial, then there is no Galois object; otherwise there is  one Galois object by Remark \ref{rmk-1}. Hence  there is at most one Galois object in this case.

Assume that $L= \langle x,y\rangle\cong \Z_4\times\Z_2$. Clearly, $\omega|_{L\times L\times L}=1$.   By definition,  $\Omega_t(x^ky^l,x^my^n)=(-1)^{lm}$,  which  is not symmetric. Then by Remark \ref{rmk-1}, $(L,\alpha)$ and  $(L,1)$ are isomorphic, where $[\alpha]$ generates $H^2(L,\K^{\times})$.
Consequently,  $(H_{b:1})^*$   has one or two Galois objects up to isomorphism.

For   $H_{b:y}$ and $H_{b:x^2y}$, $\omega(a,b,c)=(-1)^{j(k+r)s}$ and $(-1)^{(jk+jr+ir)s}$, respectively. We claim that $L\neq\langle x,ty\rangle$ or $\langle xy,ty\rangle$.    Indeed,  consider the subgroup $\langle ty\rangle\cong\Z_2$ of $L$, we have
$$\omega((ty)^i,(ty)^j,(ty)^k)=(-1)^{k[\frac{i+j}{2}]};$$
then by   Remark \ref{cocyclerepresent}, the class of $\omega|_{\langle ty\rangle\times \langle ty\rangle\times \langle ty\rangle}$ is non-trivial, which implies that $\omega|_{L\times L\times L}$ is non-trivial.

Assume that $L=\langle x\rangle\times\langle y\rangle\cong\Z_4\times\Z_2$ or $\langle x^2\rangle\times \langle y\rangle\times \langle tx\rangle \cong \Z_2\times\Z_2\times\Z_2$.
It is easy to see that $\omega|_{L\times L\times L}=1$. Now we determine the isomorphism classes of the fiber functors. If $L=\langle x\rangle\times\langle y\rangle$, then there is one isomorphism class since $\Omega_t|_{L\times L}$ is not symmetric. Now we  focus on $L=\langle x^2\rangle\times \langle y\rangle\times \langle tx\rangle$. For $H_{b:y}$, a direct computation shows $\Omega_{t\pi}=\Omega_t$ for all $ \pi\in L$  and  $\Omega_t((x^2)^iy^j(tx)^k,(x^2)^ly^m(tx)^n)=(-1)^{jn}$.
Let $(L,\alpha)$ and $(L,\beta)$ be two fiber functors, where $[\alpha], [\beta]\in H^2(L,\K^{\times})$ are determined by the sequences $(a_{12},a_{13},a_{23})$ and $(b_{12},b_{13},b_{23})$ by Proposition \ref{cohomology}, respectively. By Theorem \ref{fiberfuncondition} and Remark \ref{2cocycleabelian}, they are isomorphic if and only if there exists $g\in G$ such that the class of $\alpha\beta^g\Omega_g$ is trivial, if and only if $\alpha\beta^g\Omega_g$ is symmetric. Since \begin{gather*}\alpha\beta^t\Omega_t((x^2)^iy^j(tx)^k,(x^2)^ly^m(tx)^n)
=(-1)^{(a_{12}+b_{12})im+(a_{13}+b_{13})in+(a_{23}+b_{23}+1)jn+b_{12}km+b_{13}kn},
\end{gather*}
$\alpha\beta^t\Omega_t$ is symmetric if and only if $b_{12}=a_{12}=0$, $b_{13}+a_{13}\equiv0 \mod 2$ and $a_{23}+b_{23}+1\equiv0 \mod 2$, which means $(0,0,0)$ and $(0,0,1)$ are isomorphic, $(0,1,0)$ and $(0,1,1)$ are isomorphic. For other $tl$ $(l\in L)$, the results on equivalence classes are the same.
Consequently, there are  seven Galois objects for $(H_{b:y})^*$.
For $H_{b:x^2y}$, the proof follows the same lines.
\end{proof}

\begin{rmk}
From the proof of Lemma \ref{H(b)}, if $\omega|_{L\times L\times L}$ for $L=\langle xy,ty\rangle\cong D_8$ is trivial, then $(H_{b:1})^*$ has only one Galois object; otherwise $(H_{b:1})^*$ has two Galois objects.
\end{rmk}

\begin{lem}
Up to isomorphism,  $(H_{c:\sigma_0})^*$ has two Galois objects and $(H_{c:\sigma_1})^*$ only has trivial Galois object.
\end{lem}
\begin{proof}
By Lemma \ref{subgroups8}, the possibilities of $L$ are $\langle x\rangle\times\langle y\rangle\cong\Z_4\times \Z_2$,   $\langle tx\rangle\times\langle y\rangle\cong\Z_4\times \Z_2$. Let $a=t^qx^iy^j$,  $b=t^rx^ky^l$ and $c=t^sx^my^n$ for some $q,r,s,j,l,n\in\Z_2,i,k,m\in\Z_4$.

For $H_{c:\sigma_0}$, $\omega(a,b,c)=(-1)^{{\frac{2(ri+j)ks+i(i-1)rs}{2}}}$. Assume that $L=\langle tx\rangle\times\langle y\rangle\cong\Z_4\times\Z_2$.
We first show that the class of $\omega|_{L\times L\times L}$ is trivial. Observe that $(tx)^iy^j=t^ix^iy^{[\frac{i}{2}]+j}$; then direct computations show that $\omega|_{P\times P\times P}=1$ if $P=
\langle tx\rangle$, $\langle y\rangle$ or $\langle (tx)^2y\rangle=\langle x^2\rangle$; following the same line as the proof of the case $H_{a:y}$ in Lemma \ref{H(a,y)}, the claim follows.  Now we show that $(L,1)$ and $(L,\alpha)$ are isomorphic, where $[\alpha]$ generates $H^2(\Z_4\times\Z_2,\K^{\times})$. Let $\beta:=tx$ for short. Observe that $ty=yt$, $x^kt=tx^ky^k$ and $\beta^kt=t^{k+1}x^ky^{[\frac{k}{2}]+k}=t\beta^ky^k$. By definition,
\begin{align*}
\Omega_t(\beta^ky^l,\beta^my^n)
&=\frac{\omega(\beta^ky^{k+l},\beta^my^{m+n},t)\omega(t,\beta^ky^l,\beta^my^n)}{\omega(\beta^ky^{k+l},t,\beta^my^n)}\\
&=\frac{\omega(t^kx^ky^{[\frac{k}{2}]+k+l},t^mx^my^{[\frac{m}{2}]+m+n},t)
\omega(t,t^kx^ky^{[\frac{k}{2}]+l},t^mx^my^{[\frac{m}{2}]+n})}{\omega(t^kx^ky^{[\frac{k}{2}]+k+l},t,t^mx^my^{[\frac{m}{2}]+n})}\\
&=(-1)^{mk+[\frac{k}{2}]+k+l}.
\end{align*}
It is easy to see that $\Omega_t$ is not symmetric. Hence by Remark \ref{rmk-1},   $(L,1)$ and $(L,\alpha)$ are isomorphic.

Assume that $L=\langle x\rangle\times\langle y\rangle\cong\Z_4\times \Z_2$. Clearly, $\omega|_{L\times L\times L}=1$. We claim that $(L,1)$ and $(L,\alpha)$ are isomorphic. Indeed,  $\Omega_t(x^ky^l,x^my^n)=(-1)^{(k+l)m}$, which is not symmetric, hence the claim follows by Remark \ref{rmk-1}.
Consequently, there are two Galois objects up to isomorphism for $(H_{c:\sigma_0})^*$.

For $H_{c:\sigma_1}$, we have $\omega(a,b,c)=(-1)^{{\frac{2(ri+j)ks+i(i-1)rs}{2}}}\theta^{irs}$. We first claim that $L\neq\langle tx\rangle\times\langle y\rangle$.  Indeed, it suffices to show that  the restriction of $\omega$ on the subgroup $\langle tx\rangle\cong\Z_4$ is non-trivial.   By Remark \ref{cocyclerepresent}, there is a 2-cochain $\varphi:\langle tx\rangle\times \langle tx\rangle\rightarrow\K^{\times}$ such that $\omega d^2(\varphi)=\omega_a$ for some $0\leq a<4$. Observe that $(\omega d^2(\varphi))(tx, (tx)^j,tx)=\nu_4^a$ if $j=3$ otherwise it equals to $1$. Then  $\nu_4^a=-\theta\frac{\varphi((ty)^3,ty)}{\varphi(ty, (ty)^3)}=(-\theta)^2=-1$, that is, the class of $\omega|_{\langle tx\rangle\times\langle tx\rangle\times\langle tx\rangle}$ is non-trivial, which implies that   the claim follows.

Assume that $L=\langle x\rangle\times\langle y\rangle\cong\Z_4\times \Z_2$. Clearly, $\omega|_{L\times L\times L}=1$. Since $\Omega_t|_{L\times L}$ is not symmetric, it follows by Remark \ref{rmk-1} that  $(L,1)$ and $(L,\alpha)$ are isomorphic, where $[\alpha]$ represents $H^2(L,\K^\times)\cong\Z_2$. Consequently, there are only trivial Galois objects for $(H_{c:\sigma_1})^*$.
\end{proof}
\begin{lem}\label{lem:H{d:-1,1}}
Up to isomorphism, $(H_{d:-1,1})^*\cong H_{d:-1,1}$ has four Galois objects, $(H_{d:-1,-1})^*\cong H_{c:\sigma_0}$ and $(H_{d:1,-1})^*\cong H_{b:1}$ have  five Galois objects, the number of Galois objects of $(H_{d:1,1})^*\cong H_{d:1,1}$ ranges from six to eight.
\end{lem}
\begin{proof}
By Lemma \ref{subgroups8}, the possibilities of $L$ are $\langle xt\rangle\times\langle z\rangle\cong \Z_4\times \Z_2$,  $\langle x\rangle\times \langle y\rangle\times \langle z\rangle\cong  \Z_2\times \Z_2\times \Z_2$, $\langle xt ,yz\rangle\cong D_8$,  $\langle (xt)z,y\rangle\cong D_8$. %Let $a=t^px^iy^jz^k$, $b=t^qx^ly^mz^n$,  $c=t^rx^fy^gz^h$ and $\omega:=\omega(\tau,\sigma)$ for short.For $H_{d:-1,1}$ and   $H_{d:1,-1}$,   $\omega(a,b,c)=(-1)^{kr(l+m)+\frac{(i+j)+(-1)^{q}(j-i)}{2}lr}(-1)^{qrij}$ and $\omega(a,b,c)=(-1)^{(kl+km+ql)r}$, respectively.

For $H_{d:-1,1}$ and   $H_{d:1,-1}$,   we claim that $L= \langle x\rangle\times\langle y\rangle\times\langle z\rangle\cong\Z_2\times\Z_2\times\Z_2$. Indeed, take $H_{d:-1,1}$ for example, if $L=\langle xt\rangle\times\langle z\rangle$ or $\langle (xt)z,y\rangle$,   then consider the restriction of $\omega$ on the subgroup $\langle (xt)z\rangle\cong\Z_4$, by Remark \ref{cocyclerepresent}, there is a 2-cochain $\varphi:\Z_4\times \Z_4\to \K^\times$  such that $\omega d^2(\varphi)=\omega_a$ for some $0\leq a<4$. A direct computation shows  that $a=2$, that is,  the class of $\omega|_{\langle (xt)z\rangle\times\langle (xt)z\rangle\times\langle (xt)z\rangle}$  is non-trivial in $H^3(\Z_4,\K^{\times})$ and hence  the claim follows. If $L=\langle xt,yz\rangle$, then similar computations as before show that  the class of $\omega$ is non-trivial when restricted to the subgroup $\langle xt\rangle$. Clearly, $\omega|_{L\times L\times L}=1$ if $L= \langle x\rangle\times\langle y\rangle\times\langle z\rangle$. Therefore, the claim follows.

Now we determine the isomorphism classes of the fiber functors. For $H_{d:-1,1}$, by definition,
$$\Omega_g(x^iy^jz^k,x^ly^mz^n)=
\frac{\omega(gx^iy^jz^kg^{-1},gx^ly^mz^ng^{-1},g)\omega(g,x^iy^jz^k,x^ly^mz^n)}
{\omega(gx^iy^jz^kg^{-1},g,x^ly^mz^n)},\quad \forall g\in G.$$
Then for any $\pi\in L$,
\begin{align*}
\Omega_{t\pi}(x^iy^jz^k,x^ly^mz^n)=\Omega_t(x^iy^jz^k,x^ly^mz^n)=(-1)^{kl+km+im}.
\end{align*}

Reorganizing the order of $x^iy^jz^k$ to $z^kx^iy^j$,  by Proposition \ref{cohomology}, $\Omega_t$ corresponds to the element determined by $(1,1,1)$ in $H^2(\Z_2\times\Z_2\times\Z_2,\K^{\times})$.  Let $[\alpha],[\beta]\in H^2(\Z_2\times\Z_2\times\Z_2,\K^{\times})$   determined by the sequences $(a_{12},a_{13},a_{23})$ and $(b_{12},b_{13},b_{23})$, respectively. Observe that $\Omega_{t\pi}=\Omega_{t}$ for all $\pi\in L$. Then  by Remark \ref{2cocycleabelian},  fiber functors $(L,\alpha)$ and $(L,\beta)$ are   isomorphic, if and only if, the class of $\alpha\beta^t\Omega_t$ is trivial, if and only if, $\alpha\beta^t\Omega_t$ is symmetric. By definition of $\alpha, \beta, \beta^t$  and $\Omega_t$,
$$\alpha\beta^t\Omega_t(z^kx^iy^j,z^nx^ly^m)=
(-1)^{(a_{12}+b_{13}+1)kl+(a_{13}+b_{12}+1)km+(a_{23}+1)im+b_{23}jl}.$$
Then $\alpha\beta^t\Omega_t(z^kx^iy^j,z^nx^ly^m)$ is symmetric if and only if $b_{23}=0, a_{23}=1$, $a_{12}+b_{13}+1=a_{13}+b_{12}+1\equiv0 \mod 2$. Therefore, elements in $H^2(\mathbb{Z}_2\times\mathbb{Z}_2\times\mathbb{Z}_2,\K^{\times})$ are divided into four classes:
$$
 \{(0,0,1), (1,1,0)\},\quad \{(1,1,1), (0,0,0)\},\quad \{(1,0,1), (1,0,0)\},\quad \{(0,1,1), (0,1,0)\}.
$$
Therefore,  there are   four Galois objects up to isomorphism for $(H_{d:-1,1})^*\cong H_{d:-1,1}$. For $H_{d:1,-1}$, the proof follows the same line as for $H_{d:-1,1}$.

For $H_{d:-1,-1}$, we show that the class of $\omega|_{L\times L\times L}$ is  trivial only when $L=\langle x\rangle\times\langle y\rangle\times\langle z\rangle$ and $\langle xt\rangle\times\langle z\rangle$.
If $L=\langle (xt)z,y\rangle$ or $\langle xt,yz\rangle$,   then  $\omega((txyz)^i,(txyz)^j,(txyz)^k)=(-1)^{k[\frac{i+j}{2}]}$, and hence  by Remark \ref{cocyclerepresent}    the class of  $\omega$  is not trivial in $H^3(\Z_2,\K^{\times})$  when restricted to the subgroup $\langle txyz\rangle\cong\Z_2$, which implies that the class of $\omega|_{L\times L\times L}$ is not trivial in $H^3(L,\K^{\times})$.
If $L=\langle x\rangle\times\langle y\rangle\times\langle z\rangle$, then completely analogous to the case  $H_{d:-1,1}$,   there exist four isomorphism classes.
If $L=\langle xt\rangle\times\langle z\rangle$, then similar to the proof of Lemma \ref{H(a,y)},  the class of  $\omega|_{L\times L\times L}$ is trivial and direct computations show that $\Omega_{t}$ is not symmetric, which implies that $(L,1)$ and $(L,\alpha)$ are isomorphic. Consequently, there are   five Galois objects for $(H_{d:-1,-1})^*$.

For $H_{d:1,1}$, we   show that the class of $\omega|_{L\times L\times L}$ is  trivial when $L=\langle x\rangle\times\langle y\rangle\times\langle z\rangle$ and $\langle xt\rangle\times\langle z\rangle$. Similar to  the case  $H_{d:-1,-1}$,   the fiber functors $(L,\alpha)$ admit   five or one  isomorphism class if $L=\langle x\rangle\times\langle y\rangle\times\langle z\rangle$ or $\langle xt\rangle\times\langle z\rangle$, respectively.  For $L=\langle xt, yz\rangle\cong D_8$  or $ \langle (xt)z, y\rangle\cong D_8$, a direct computation shows that the  class of $\Omega_t|_{L\times L}$ is non-trivial in both cases, then there are at most two Galois objects as in Lemma \ref{H(b)}. Consequently, the number of Galois objects of $(H_{d:1,1})^*\cong H_{d:1,1}$ ranges from six to eight.
\end{proof}

\begin{rmk}
From the proof of Lemma \ref{lem:H{d:-1,1}}, if $\omega|_{\langle xt, yz\rangle\times\langle xt, yz\rangle\times\langle xt, yz\rangle}$ and $\omega|_{\langle (xt)z, y\rangle\times\langle (xt)z, y\rangle\times\langle (xt)z, y\rangle}$ are trivial, then $H_{d:1,1}$ has eight Galois objects; if $\omega|_{\langle xt, yz\rangle\times\langle xt, yz\rangle\times\langle xt, yz\rangle}$ and $\omega|_{\langle (xt)z, y\rangle\times\langle (xt)z, y\rangle\times\langle (xt)z, y\rangle}$ are not trivial, then $H_{d:1,1}$ has six Galois objects; otherwise, $H_{d:1,1}$ has seven Galois objects.

\end{rmk}

\begin{lem}
$(H_{B:1})^*\cong H_{C:1} $ and $(H_{B:X})^*\cong H_E $ have only trivial Galois objects up to isomorphism.
\end{lem}
\begin{proof}
By Lemma \ref{subgroups8}, the possibilities of $L$ are $\langle x\rangle\times\langle yt\rangle\cong\Z_4\times \Z_2$, $\langle x^2\rangle\times \langle tx\rangle\times \langle yt\rangle\cong \Z_2\times \Z_2\times \Z_2$  $\langle x,t(yt)\rangle=\langle x,y\rangle\cong D_8$, $\langle x(yt),t(yt)\rangle=\langle x(yt),y\rangle\cong D_8$.

For $ H_{B:1}$, we first claim that $L=\langle x,y\rangle\cong D_8$. If   $L=\langle x^2\rangle\times\langle tx\rangle\times\langle yt\rangle$ or $\langle x\rangle\times\langle yt\rangle\cong\Z_4\times\Z_2$, then consider the subgroup $\langle tx^2y\rangle\cong\Z_2$ of $L$, we have
$$
\omega((tx^2y)^i,(tx^2y)^j,(tx^2y)^k)=(-1)^{k[\frac{i+j}{2}]};
$$
by Remark \ref{cocyclerepresent},   the class of $\omega$  is not trivial in $H^3(\Z_2,\K^\times)$ when restricted to subgroup   $\langle tx^2y\rangle$, which implies that the class of  $\omega|_{L\times L\times L}$ is non-trivial. If $L=\langle x(yt),y\rangle\cong D_8$, then  a direct computation shows that  restriction of $\omega$ on  $\langle x(yt)\rangle\cong\Z_4$ is a non-trivial class by Remark \ref{cocyclerepresent}. Clearly, $\omega|_{L\times L\times L}=1$ if  $L=\langle x,y\rangle\cong D_8$. Therefore, the claim follows.

Now we claim that $(L,1)$ and $(L,\beta)$ are isomorphic, where $[\beta]$ represents $ H^2(D_8,\K^{\times})$.    Observe that $\Omega_t(x^ky^l,x^my^n)=\xi^{lm}$. If $\Omega_t=d(f)$ for some morphism $f:L\to\K^\times$, then  $f(x^k)=f(x)^k$ and $f(y)^2=f(y^2)=f(x)^4=1$. Observe that $1=\Omega_t(x^3,y)$. Then  $f(x^3)f(y)=f(x^3y)=f(x)^3f(y)$. Since $\xi^{-1}=\Omega_t(y,x)=f(y)f(x)f(yx)^{-1}$, it follows that $-f(x)^2=f(yx)^2=f(x^3y)^2=f(x)^6f(y)^2=f(x)^2$, a contradiction. Therefore, the class of $\Omega_t|_{L\times L}$ is not trivial and then the claim follows by Remark \ref{rmk-1}. Consequently, there are    only trivial Galois objects for $(H_{B:1})^*\cong H_{C:1}$.

For $H_{B:X}$, the proof follows the same line as for $H_{B:1}$.
\end{proof}

\begin{lem}Up to isomorphism $H_{B:1}\cong (H_{C:1})^*$ has three Galois objects, and $H_{C:\sigma_1}$ has only   trivial Galois objects.
\end{lem}
\begin{proof}
 By Lemma \ref{subgroups8}, the possibilities of $L$ are $\langle ty\rangle\cong\Z_8$ and $\langle x,y\rangle\cong D_8$. %Let $a=t^qx^iy^j,b=t^rx^ky^l,c=t^sx^my^n$ for some $0\leq q,j,r,k,s,n<2,0\leq i,k,m<4$ and $\omega:=\omega(\tau,\sigma)$.$\omega(a,b,c)=\mu^{(r-2k)js}$,

For $(H_{C:\sigma_1})^*\cong H_{C:\sigma_1}$, we first claim that $L\neq\langle ty\rangle\cong\Z_8$.  Indeed, if the class of $\omega|_{L\times L\times L}$ is  trivial, then by Remark \ref{cocyclerepresent},  there must be a 2-cochain $\varphi: L\times L\to \K^\times$ such that $\omega d^2(\varphi)=\omega_a$ for some $0\leq a<8$. Then $\omega((ty),(ty)^j,(ty))=\nu_8^a$ if $j=7$ otherwise it is equal to 1. In particular, $\nu_8^a=\mu^{-1}\frac{\varphi((ty)^7,(ty))}{\varphi((ty),(ty)^7)}$, where $\mu$ is a primitive 8-th root of unity, and induction on $j$ shows that $\nu_8^a=\mu^2\neq1$.

Assume that  $L=\langle x,y\rangle\cong D_8$. Clearly, $\omega|_{L\times L\times L}=1$. Now we determine the isomorphism classes of fiber functors.    By definition, $\Omega_t(x^ky^l,x^my^n)=\xi^{l(n-m)}$.  We claim that  the class of $\Omega_t$  is not trivial in  $H^2(D_8,\K^{\times})$. Indeed, if $\Omega_t=d(f)$ for some morphism $f:D_8\to \K^{\times}$, then from $\Omega_t(y,y)=1=\Omega_t(x,y)$, we have $f(y)^2=1$ and $f(xy)=f(x)f(y)$; since $\Omega_t(xy,y)=\xi$, it follows that $\xi f(x)=f(xy)f(y)$, $-f(x)^2=f(xy)^2f(y)^2=f(x)^2$, a contradiction. Consequently, there is only one Galois object up to isomorphism for $H_{C:\sigma_1}$ by Remark \ref{rmk-1}.

For   $H_{C:1} $,    $\omega(\tau,\sigma)$ is trivial, so is $\Omega_g$ for $\forall g\in G$. Hence by Theorem \ref{fiberfuncondition}, Lemma \ref{subgroups8} and Proposition \ref{Schurmultiplier}, there are three Galois objects for $(H_{C:1})^*\cong H_{B:1}$.
\end{proof}
\begin{lem}There are two Galois objects up to isomorphism for $(H_E)^*\cong H_{B:X}$.
\end{lem}
\begin{proof}
In this case,  $\omega(\tau,\sigma)$ is trivial. Then by Lemma \ref{subgroups8},    $L=\langle ty\rangle\cong\Z_8$ or $\langle x, y\rangle\cong Q_8$. Consequently, the lemma follows by  Proposition \ref{Schurmultiplier} and Theorem \ref{fiberfuncondition}.
\end{proof}
In summary, we have   the following theorem:
\begin{thm}\label{thmAmain}
Let $H$ be a semisimple Hopf algebra of dimension $16$ in \cite[Table\,1]{Ka}, up to isomorphism the number of Galois objects of $H^*$ is given as follows:
\begin{enumerate}
\item  $(H_{C:\sigma_1})^*\cong H_{C:\sigma_1} $, $(H_{{c:\sigma_{1}}})^*$, $(H_{B:1})^*\cong H_{C:1}$ and $(H_{B:X})^*\cong H_E$ only have one  trivial Galois object.
\item  $(H_E)^*\cong H_{B:X}$ and $(H_{{c:\sigma_{0}}})^*$ have two Galois objects.
\item  $(H_{C:1})^*\cong H_{B:1}$ and $(H_{a:y})^*$ have  three Galois objects.
\item  $(H_{d:-1,1})^*\cong H_{d:-1,1}$ has four Galois objects.
\item $(H_{d:-1,-1})^*\cong H_{{c:\sigma_{1}}}$ and $(H_{d:1,-1})^*\cong H_{{b:1}}$  have   five Galois objects.
\item  $(H_{{b:x^2y}})^*$ and $(H_{b:y})^*$ have seven Galois objects.
\item $(H_{a:1})^*$ and $(H_{b:1})^*$ have one or two Galois objects, the number of Galois objects of $(H_{d:1,1})^*\cong H_{d:1,1}$ ranges from six to eight.
 \end{enumerate}
\end{thm}
\begin{rmk}
Masuoka has already proved that $\mathcal{B}_{16}\cong H_{C:\sigma_1}$ has only trivial Galois objects and $\mathcal{A}_{16}\cong H_{B:1}$ has  three   Galois objects up to isomorphism, see \cite[Theorem 4.8]{Ma2} for  details.
\end{rmk}
\section{Cocycle deformations of semisimple Hopf algebras of dimension $16$}\label{secCocycle}

In this section, we determine  cocycle deformations of semisimple Hopf algebras of dimension $16$ listed in \cite[Table\,1]{Ka} using their Frobenius-Schur indicators and  Grothendieck rings. The Grothendieck rings of these Hopf algebras have been determined by Kashina $($see \cite[Table 1]{Ka}$)$ and the second indicators of irreducible representations  have been computed in \cite{KMM}.

We first determine   twist equivalence classes of these Hopf algebras.
\begin{lem}\label{indicators}
The Hopf algebras $H_{a:1}$, $H_{a:y}$, $H_{b:y}$, $H_{b:{x^2y}}$, $(H_{c:{\sigma_1}})^*\cong H_{d:-1,-1}$ and  $(H_{b:1})^*\cong H_{d:1,-1}$ have different Frobenius-Schur indicator  sets and hence  they are pairwise Drinfeld twist inequivalent.
\end{lem}
\begin{proof}
Clearly, it is not enough to distinguish all the Drinfeld twist equivalence classes via the second indicators. We shall prove the lemma by computing   higher Frobenius-Schur indicators. A direct computation shows that the third indicators are not sufficient to distinguish them. Here we omit the computation details, because  the expression of the Sweedler power of the integral is too long.
Let $\chi_{1}$, $\chi_2$ be the characters of $2$-dimensional irreducible representations of these Hopf algebras, respectively.

The fourth Sweedler power of integral $\Lambda$ with $\varepsilon(\Lambda)=1$ is given by
\begin{gather*}
\Lambda^{[4]}=\left\{
  \begin{array}{ll}
    \frac{3+y}{4}, & \hbox{$H\cong H_{a:1}$ or $H_{a:y}$;} \\
    \frac{3}{2}, & \hbox{$H\cong H_{b:y}$;}\\
     1, & \hbox{$H\cong H_{d:1,-1}$;} \\
      \frac{1+xy}{2}, & \hbox{$H\cong H_{d:-1,-1}$.}
  \end{array}
\right.
\end{gather*}

%For $H=H_{a:1}$ and $H_{a:y}$, we have $y\mapsto\left(
%                                                   \begin{array}{cc}
%                                                     -1 & 0 \\
%                                                     0 & -1 \\
%                                                   \end{array}
%                                                 \right)
%$ for the $2$-dimensional irreducible representations.

For $H_{a:1}$, the indicators are: $\nu_2(\chi_{1})=-1,~\nu_2(\chi_{2})=1$;~$\nu_3(\chi_{1})=\nu_3(\chi_{2})=0$; $\nu_4(\chi_{1})=\nu_4(\chi_{2})=1$.

For $H_{a:y}$, the indicators are: $\nu_2(\chi_{1})=\nu_2(\chi_{2})=1$;~$\nu_3(\chi_{1})=\nu_3(\chi_{2})=0$; $\nu_4(\chi_{1})=\nu_4(\chi_{2})=1$.

For $H_{b:y}$,  the indicators are: $\nu_2(\chi_{1})=1,~\nu_2(\chi_{2})=-1$;~$\nu_3(\chi_{1})=\nu_3(\chi_{2})=0$; $\nu_4(\chi_{1})=\nu_4(\chi_{2})=3$.

For $H_{b:{x^2y}}$, the indicators are: $\nu_2(\chi_{1})=\nu_2(\chi_{2})=-1$.

For $H_{d:-1,-1}$, the indicators are: $\nu_2(\chi_{1})=\nu_2(\chi_{2})=1$;
$\nu_3(\chi_{1})=\nu_3(\chi_{2})=\nu_4(\chi_{1})=\nu_4(\chi_{2})=0$.
% as $x\mapsto\left(
%  \begin{array}{cc}
%    1 & 0 \\
%    0 & -1 \\
%  \end{array}
%  \right)
%$,
%$y\mapsto\left(
%\begin{array}{cc}
%  -1 & 0 \\
%  0 & 1 \\
%\end{array}
%\right)
%$ for both irreducible representations.

For $H_{d:1,-1}$ the indicators are: $\nu_2(\chi_{1})=\nu_2(\chi_{2})=1$,~$\nu_3(\chi_{1})=\nu_3(\chi_{2})=0$,
~$\nu_4(\chi_{1})=\nu_4(\chi_{2})=2$.

Consequently,  these Hopf algebras have different Frobenius-Schur indictor sets. Then by Theorem \ref{gaugeinvariant}, they are pairwise Drinfeld twist inequivalent.
\end{proof}

\begin{thm}\label{ThmTwist}
Let $H$ and $K$ be  two non-isomorphic semisimple Hopf algebras of dimension $16$ listed in \cite[Table\,1]{Ka}. Then $H$ and $K$ are twist inequivalent.
\end{thm}
\begin{proof}
From the second Frobenius-Schur indicators of the $2$-dimensional irreducible representations \cite{KMM}, we obtain six possible twist equivalence classes:
\begin{gather*}
\{H_{a:1}, H_{b:1}, H_{B:X}\},  \{H_{{b:x^2y}}\},  \{H_{{c:\sigma_{0}}},H_{{c:\sigma_{1}}}\}, \{H_{C:1},H_{{C:\sigma_{1}}}\},\\ \{H_{E}\}, \{H_{b:y}, H_{a:y}, H_{B:1}, H_{d:1,1},  H_{d:-1,1},H_{d:1,-1}, H_{d:-1,-1}\}.
\end{gather*}

Recall that the Grothendieck ring is a gauge invariant. From the Grothendieck rings of semisimple Hopf algebras listed in \cite[Table 1]{Ka}, the possible classes $\{H_{{c:\sigma_{0}}},H_{{c:\sigma_{1}}}\}$, $\{H_{a:1},H_{b:1},H_{B:X}\}$ and $\{H_{C:1},H_{{C:\sigma_{1}}}\}$ are pairwise twist inequivalent, and the last possible class is divided into three possible equivalence classes:
\begin{gather*}
 \{H_{B:1}\}, \{H_{b:y}, H_{a:y}, (H_{b:1})^*, (H_{{c:\sigma_{1}}})^*\},
\{H_{d:-1,1},~H_{d:1,1}\}.
\end{gather*}
 The Hopf algebras in the second possible class are pairwise twist inequivalent by Theorem \ref{gaugeinvariant} and Lemma \ref{indicators}. Observe that $H_{d:1,1}=\K[D_8\times \mathbb{Z}_2]^J$, $H_{d:-1,1}=H_8\otimes \K[\mathbb{Z}_2]$, $H_{d:1,-1}\cong (H_{{b:1}})^*$, $H_{d:-1,-1}\cong (H_{{c:\sigma_{1}}})^*$, and $\K[D_8\times \mathbb{Z}_2]^J$ is triangular while $H_8\otimes \K[\mathbb{Z}_2]$ is not triangular. Then the Hopf algebras in the third  class are pairwise twist inequivalent.
\end{proof}

Next we determine whether these  Hopf algebras in \cite[Table\,1]{Ka} are Drinfeld twists of group algebras. By Theorem \ref{twistings}, semisimple Hopf algebras with the same Grothendieck ring are differed by a 2-pseudo-cocycle. Therefore, it suffices to distinguish  a $2$-pseudo-cocycle is a Drinfeld twist or not.
\begin{lem}\label{groupG5G6}The Hopf algebras $H_{a:1}$, $H_{a:y}$, $H_{b:y}$, $H_{b:{x^2y}}$, $(H_{c:{\sigma_1}})^*$ and $(H_{b:1})^*$ with the same Grothendieck ring $Gr(G_5)\cong Gr(G_6)$ are not twist deformations of the group algebra $\K[G_{5}]$ or $\K[G_{6}]$.
\end{lem}
\begin{proof}
We prove the statements by showing that their Frobenius-Schur indicator sets are different from those of $G_5$ and $G_6$. For  $G_5:=\langle a,b\mid a^4=b^4=1, ba=a^{-1}b\rangle$, there are two self-dual $2$-dimensional irreducible representations $\pi_1,\pi_2$ with characters $\chi_{1}$, $\chi_2$ given by
\begin{align*}
\pi_1(a)=\left(
           \begin{array}{cc}
             0 & 1 \\
             1 & 0 \\
           \end{array}
         \right),
\pi_1(b)=\left(\begin{array}{cc}
           \theta & 0 \\
           0 & -\theta
         \end{array}\right);
%\end{align*}
%\begin{align*}
\pi_2(a)=\left(\begin{array}{cc}
            0 & \theta \\
            \theta & 0
          \end{array}\right),
\pi_2(b)=\left(\begin{array}{cc}
            \theta & 0 \\
            0 & -\theta
          \end{array}\right);
\end{align*}
where $\theta$ is a $4$-th root of unity and the integral element $\Lambda=\frac{1}{16}\Sigma^{3}_{i,j=0}a^ib^j$. Then a direct computation shows that the higher Frobenius-Schur indicators are: $\nu_2(\chi_{1})=1, \nu_2(\chi_{2})=-1$; $\nu_3(\chi_{1})=\frac{-\theta}{4}, \nu_3(\chi_{2})=0$; $\nu_4(\chi_{1})=\nu_4(\chi_{2})=2$.

For  $G_6:=\langle a,~b,~c\mid a^4=b^2=c^2=1,~bab=ac\rangle$, there are two self-dual $2$-dimensional irreducible representations $\pi_1,~\pi_2$ with characters $\chi_{1},~\chi_2$ given by
\begin{align*}
\pi_1(a)=\left(
           \begin{array}{cc}
             0 & -1 \\
             1 & 0 \\
           \end{array}
         \right),
\qquad\pi_1(b)=\left(\begin{array}{cc}
           0 & 1 \\
           1 & 0
         \end{array}\right),
 \qquad\pi_1(c)=\left(\begin{array}{cc}
           -1 & 0 \\
           0 & -1
         \end{array}\right);
\end{align*}
\begin{align*}
\pi_2(a)=\left(\begin{array}{cc}
            0 & -\theta \\
            \theta & 0
          \end{array}\right),
\qquad\pi_2(b)=\left(\begin{array}{cc}
            0 & 1 \\
            1 & 0
          \end{array}\right),
\qquad\pi_2(c)=\left(\begin{array}{cc}
           -1 & 0 \\
           0 & -1
         \end{array}\right);
\end{align*}
and the integral element $\Lambda=\frac{1}{16}\sum\limits_{\substack{0\leq i\leq3,\\ 0\leq j,k\leq1}}a^ib^jc^k$. Then a direct computation shows that higher Frobenius-Schur indicators of $\K[G_6]$ are: $\nu_2(\chi_{1})=\nu_2(\chi_{2})=1$;~$\nu_3(\chi_{1})=\nu_3(\chi_{2})=0$; $\nu_4(\chi_{1})=\nu_4(\chi_{2})=2$.

By Theorem \ref{gaugeinvariant} and Lemma \ref{indicators},  the proof is completed.
\end{proof}

\begin{thm}\label{thmBmain1}
Except the Hopf algebras $H_E\cong \K[G_2]^J,~H_{C:1}\cong \K[D_{16}]^J$ and $H_{d:1,1}\cong \K[D_8\times\mathbb{Z}_2]^J$, the semisimple Hopf algebras of dimension $16$ in \cite[Table\,1]{Ka} are not twist equivalent to  group algebras.
\end{thm}
\begin{proof}From \cite[Table\,1]{Ka}, there are $7$ isomorphism classes of Grothendieck rings, that is, $Gr(G_1)$, $Gr(G_2)$, $Gr(G_5)$, $Gr(G_7)$, $Gr(D_{16})$, $Gr(D_8\times\mathbb{Z}_2)$ and the only non-commutative fusion ring $K_{5,5}$. Moreover, $H_E\cong \K[G_2]^J,~H_{C:1}\cong \K[D_{16}]^J$ and $H_{d:1,1}\cong \K[D_8\times\mathbb{Z}_2]^J$.

Now we claim that  the Hopf algebras $H$  in \cite[Table\,1]{Ka} are not twist deformations of group algebras  except for $H_E$, $H_{C:1}$ and $H_{d:1,1}$. By Theorem \ref{twistings}, it suffices to show $H$ is not twist equivalent to $\K[G_i]$ when $Gr(H)\cong Gr(G_i)$.

%{\color{red}If $H\cong H_{{c:\sigma_{0}}}$}, then $Gr(H)\cong Gr(G_1)$, where $G_{1}:=\langle a, b\mid a^8=b^2=1,~ba=a^5b\rangle$. A direct computation shows that the third indicators are not sufficient to distinguish them. For  $k[G_1]$, $\Lambda^{[4]}=\frac{1}{8}\Sigma^{7}_{i=0}a^{4i}
%=\frac{1}{2}(1+a^4)$, and then the fourth indicators of the 2-dimensional representations are zero, since  $\pi_1(a)=\left(
%   \begin{array}{cc}
 %    \xi & 0 \\
 %    0 & -\xi \\
 %  \end{array}
 %\right)$
% and $\pi_2(a)=\left(
%    \begin{array}{cc}
%      \xi^3 & 0 \\
%     0  & -\xi^3 \\
%    \end{array}
%  \right)
%$,   where $\xi$ is a $8$-th primitive root of unity and $\pi_1,\pi_2$ denote the two $2$-dimensional irreducible representations of $k[G_1]$. For  $H_{{c:\sigma_{0}}}$, $\Lambda^{[4]}=\frac{3}{2}$ and the fourth indicators are $3$. Therefore, $H_{{c:\sigma_{0}}}$ is not Drinfeld twist equivalent to $k[G_1]$.

%{\color{red}If $H\cong H_{C:{\sigma_1}}$}, then $Gr(H)\cong Gr(D_{16})=Gr(Q_{16})$. $H_{C:{\sigma_1}}$ is not Drinfeld twist equivalent to $k[D_{16}]$ or $k[Q_{16}]$. Indeed, $H_{C:1}$ is the only Drinfeld twist deformation of $k[D_{16}]$, and $k[Q_{16}]$ does not have any non-trivial twist deformation \cite{Ma2}.

%{\color{red} H_{b:1}$}Since the second indicators of $H_{b:1}$ are $-1$ and $1$ respectively, and the second indicators of $D_8\times \mathbb{Z}_2, Q_8\times \mathbb{Z}_2$ are $1$,   $H_{b:1}$ is not a Drinfeld twist deformation of them.

If $H\cong H_{d:-1,1}$, then $H_{d:-1,1}\cong H_8\otimes \K[\Z_2]$ and $Gr(H)\cong Gr(D_8\times \mathbb{Z}_2)\cong Gr(Q_8\times \mathbb{Z}_2)$. This case has been stated in Theorem \ref{ThmTwist}.
If $H\cong H_{C,\sigma_1}$, then $Gr(H_{C,\sigma_1})\cong Gr(D_{16})=Gr(Q_{16})$. Since $H_{C:1}$ is the only Drinfeld twist deformation of $\K[D_{16}]$, and $\K[Q_{16}]$ does not have any non-trivial twist deformation \cite{Ma2}, it follows that $H_{C:{\sigma_1}}$ is not Drinfeld twist equivalent to $\K[D_{16}]$ or $\K[Q_{16}]$.
%{\color{red}If $H\cong H_{B:1}$ or $H_{B:X}$}, then $Gr(H)\cong K_{5,5}$. $H$ is not a Drinfeld twist deformation of a group algebra. In fact, the fusion ring $K_{5,5}$ is not commutative.
By Theorem $\ref{thmAmain}$,  the dual Hopf algebras of  $H_{c,\sigma_1}$, $H_{B,1}$ and $H_{B,X}$ have only trivial Galois objects. Therefore, they are not twist equivalent to group  algebras. For    $H_{a,1}$, $H_{b,1}$ and $H_{c,\sigma_0}$, their   Frobenius-Schur indicator  sets are different from group algebras by Lemma \ref{groupG5G6} and \cite{KMM}. Consequently, the proof is completed by Theorem \ref{gaugeinvariant}.
\end{proof}
\begin{rmk}Note that $H_{d:-1,1}\cong H_8\otimes \K[\mathbb{Z}_2]$. It turns out in \cite{NS1} that  $H_8$, $\K[Q_8]$ and  $\K[D_8]$ have different Frobenius-Schur indicator  sets; if $H\cong H_{B:1}$ or $H_{B:X}$, then $Gr(H)\cong K_{5,5}$, which is non-commutative \cite{Ka}.  Then the statements for $H_{d:-1,1}$ $H_{B:1}$ and $H_{B:X}$ in Theorem $\ref{thmBmain1}$ follow.
\end{rmk}
%From the classification results of triangular semisimple Hopf algebras \cite{EG}, we have:
%\begin{cor}Except  for the Hopf algebras $H_E,~H_{C:1}$ and $H_{d:1,1}$, the  Hopf algebras listed in \cite[Table\,1]{Ka} are not triangular.\end{cor}

Finally, we determine  cocycle deformations of semisimple Hopf algebras  listed in \cite[Table\,1]{Ka}.
\begin{lem}\label{lemCo}
The Hopf algebras $H_{B:1},~H_{B:X},~H_{d:1,1}$ listed in \cite[Table\,1]{Ka} have only one non-trivial cocycle deformation.
\end{lem}
\begin{proof}
Observe that $H_{E}=\K[G_2]^{{J_1}}$, $H_{C:1}=\K[D_{16}]^{{J_2}}$, $H_{B:X}\cong H_{E}^*$ and $H_{B:1}\cong (H_{C:1})^*$.   Then $H_{B:X}=(\K^{{G_2}})^{\sigma}$ and $H_{B:1}=(\K^{{D_{16}}})^{\mu}$ for some non-trivial $2$-cocycles $\sigma$, $\mu$.  By Theorem \ref{thmAmain}, $H_{B:X}$ has  two Galois objects. Hence $H_{B:X}$ has only one non-trivial $2$-cocycle deformation. Since $\K[D_{16}]$  has only one non-trivial Drinfeld twist \cite{Ma2}, it follows that $H_{B:1}$ has only one non-trivial cocycle deformation.

Since $H_{d:1,1}$ is self-dual and   isomorphic to $\K[D_8\times \mathbb{Z}_2]^J$ for some Drinfeld twist $J$,  it follows that $H_{d:1,1}\cong (\K^{D_8\times \mathbb{Z}_2})^{\sigma}$ for some non-trivial 2-cocycle $\sigma$. By Theorem \ref{thmBmain1}, among the  three Hopf algebras with  Grothendieck ring  $Gr([D_8\times \mathbb{Z}_2])$,  $H_{d:1,1}$ is the only twist deformation of $\K[D_8\times \mathbb{Z}_2]$, which implies that $H_{d:1,1}$ has only one non-trivial $2$-cocycle deformation.
\end{proof}

\begin{thm}\label{thmBmain}
The  Hopf algebras $H_{B:1},~H_{B:X},~H_{d:1,1}$ have only one non-trivial cocycle deformation, and others in \cite[Table\,1]{Ka}  have only trivial  cocycle deformations.
\end{thm}
\begin{proof}
Let $H$ be a semisimple Hopf algebra   listed in \cite[Table\,1]{Ka}. If $H$ admits a non-trivial cocycle deformation $\sigma$, then  $(H^\sigma)^*\cong (H^*)^J$, that is, $H^*$ admits a non-trivial Drinfeld twist $J$. By Theorems  \ref{ThmTwist}  and  \ref{thmBmain1}, $(H^*)^J $ is isomorphic to $\K[G_2]$, $\K[D_{16}]$, or $\K[D_8\times \mathbb{Z}_2]$, equivalently $H$ is a cocycle deformation of a dual group algebra. Then the theorem follows by    Theorems $\ref{ThmTwist}$ $\&$ $\ref{thmBmain1}$ and Lemma \ref{lemCo}.
\end{proof}
\section*{Acknowledgements}
The authors would like to appreciate the continued encouragement and help from their PhD supervisor Prof. Naihong Hu. This article started after reading \cite{CMNVW}, we are grateful to the authors in \cite{CMNVW}, specially thanks  to Prof. Chelsea Walton for her valuable suggestions and comments on earlier version of the manuscript. The authors thank the referees for pointing errors out and their comments on improving this paper.
%\end{CJK}

\end{document}